\documentclass[12pt]{iopart}

\usepackage{iopams}  
\usepackage{graphicx}
\usepackage{epstopdf}
\usepackage{float}
\usepackage{algorithm,algorithmic}
\usepackage{diagbox}
\usepackage{amsthm}
\usepackage{longtable}
\usepackage{booktabs}
\usepackage{pdflscape}
\usepackage{rotating}
\newtheorem{theorem}{Theorem}
\newtheorem{lemma}{Lemma}

\newtheorem{definition}{Definition}
\newtheorem{assumption}{Assumption}
\eqnobysec
\begin{document}

\title[Solving inverse problem with BIAN]{BIAN: A Deep Learning Method to Solve Inverse Problems Using Only Boundary Information}

\author{Feng Chen\footnote{email: chenf1896@outlook.com}, Kegan Li, Yiran Meng, Zhiyi Xiao and Pengqi Wu}

\address{School of Electircal Engineering, Xi’an Jiaotong University, Xi’an 710049, China}
\ead{lkg818@stu.xjtu.edu.cn}

\vspace{10pt}
\begin{indented}
\item[]September 2024
\end{indented}

\begin{abstract}
Over the past years, inverse problems in partial differential equations have garnered increasing interest among scientists and engineers. However, due to the lack of conventional stability, nonlinearity and non-convexity, these problems are quite challenging and difficult to solve. In this work, we propose a new kind of neural network to solve the coefficient identification problems with only the boundary information. In this work, three networks has been utilized as an approximator, a generator and a discriminator, respectively. This method is particularly useful in scenarios where the coefficients of interest have a complicated structure or are difficult to represent with traditional models. Comparative analysis against traditional coefficient estimation techniques demonstrates the superiority of our approach, not only handling high-dimensional data and complex coefficient distributions adeptly by incorporating neural networks but also eliminating the necessity for extensive internal information due to the relationship between the energy distribution within the domain to the energy flux on the boundary. Several numerical examples have been presented to substantiate the merits of this algorithm including solving the Poisson equation and Helmholtz equation with spatially varying and piecewise uniform medium.

\end{abstract}
\submitto{\IP}
\vspace{2pc}
\maketitle

\section{Introduction}

In mathematics, science, and engineering, a forward problem involves predicting or calculating the behavior or response of a system under specific conditions based on known system coefficients or input conditions. Forward problems are typically characterized by well-defined coefficients and conditions, allowing for precise mathematical representations and solutions. Due to their predictability and stability, accurate modeling of physical problems is achievable. With accurate models, precise solutions to forward problems can be obtained through model-driven approaches. These approaches rely heavily on established theoretical models that encapsulate the governing laws and principles relevant to the system under study. By implementing these models, scientists and engineers can simulate scenarios and predict outcomes with high accuracy. Model-driven solutions are particularly effective when the system dynamics are well-understood and the governing equations are well-defined, allowing for deterministic predictions and controlled manipulations. This approach is extensively used in fields such as physics, engineering, and economics, where reliable models form the basis for analysis, design, and optimization tasks.

Contrasting with forward problems, inverse problems specifically refer to deducing the causes or properties of a system from observed outcomes\cite{asterParameterEstimationInverse2005}. Due to insufficient or ambiguous information about the physical system, establishing accurate physical models for inverse problems presents formidable challenges. Consequently, applying model-driven methods to solve inverse problems proves ineffective. Inverse problems often lack direct and clear relationships between inputs and outputs, making the standard approach of using well-defined physical laws and deterministic models less applicable. Thus, alternative approaches are needed to address the indeterminate nature of inverse problems. Researchers resort to data-driven methods to tackle inverse problems\cite{arridgeSolvingInverseProblems2019}, which significantly differ from the straightforward methodologies used in solving forward problems. The ill-posed nature of inverse problems makes resolving them with data-driven methods particularly challenging.

Inverse problems have wide-ranging applications across various fields, demonstrating their significant research value. In medical imaging, for example, techniques such as computed tomography (CT), magnetic resonance imaging (MRI)\cite{Adler_2017}, and ultrasound imaging rely on solving inverse problems to reconstruct images of the interior of the body from external measurements. In geophysical exploration, inverse problems are employed to infer the Earth's subsurface properties from surface measurements, aiding in the search for natural resources such as oil, gas, and minerals\cite{reichsteinDeepLearningProcess2019}. In financial market analysis, inverse problems help in modeling and predicting market behavior based on observed economic indicators\cite{OTA2023100353}. Other applications include nondestructive testing, where inverse problems are utilized to detect flaws in materials and structures, and environmental science, where they assist in understanding and predicting environmental changes based on observed data. These diverse applications highlight the critical role of inverse problems in advancing technology and scientific understanding.

Leveraging artificial intelligence (AI) to solve inverse problems offers significant advantages, especially compared to traditional numerical algorithms such as finite element methods (FEM), finite difference methods (FDM) and boundary element methods (BEM). Traditional numerical methods primarily address well-posed boundary value problems, relying on precise descriptions of boundary conditions and model coefficients\cite{evans2022partial}. However, when data is introduced, these methods often transform the original problem into an overdetermined system and lack interfaces to flexibly incorporate data during the solving process. This limitation makes traditional numerical methods less effective in handling complex and incomplete data.

In contrast, AI demonstrates considerable flexibility and effectiveness in solving inverse problems. AI can seamlessly integrate data-driven and model-driven approaches by incorporating numerical information into objective functions, constraints, and optimization algorithms. Techniques such as deep learning and reinforcement learning enable AI to handle large and complex datasets. Even in the absence of complete models, AI can effectively infer system coefficients or input conditions by learning patterns and relationships within the data. This hybrid approach leverages both data and models, providing robust solutions to inverse problems that are otherwise challenging for traditional methods.

Currently, research on utilizing artificial intelligence (AI) to solve inverse problems focuses on three main methods\cite{Nganyu}:

\begin{itemize}
      \item {\bf Physics-Informed Neural Networks (PINNs).} 
      The approach to solving inverse problems using PINNs involves directly incorporating physical laws into the loss function of the neural network\cite{raissiPhysicsinformedNeuralNetworks2019}. This integration allows the network to learn the solution that not only fits the observed data but also satisfies the underlying physical principles. By minimizing this physics-informed loss function, PINNs can accurately approximate the system's unknown coefficients or inputs, combining the strengths of data-driven approaches with the rigor of model-driven methods. This approach ensures that the solution process is consistent with the actual physical phenomena, providing robust and reliable results even in the presence of incomplete or noisy data.
      \item {\bf Deep Ritz Method (DRM).} 
      The Deep Ritz Method (DRM) addresses inverse problems by leveraging variational principles, where the solution to a partial differential equation (PDE) is obtained by minimizing an energy functional\cite{eDeepRitzMethod2018}. By training the neural network to minimize the energy functional, DRM effectively approximates the solution to the PDE, capturing the underlying physical phenomena. This method combines the strengths of neural networks in handling complex data with the rigor of variational principles, providing a robust framework for solving inverse problems.
      \item {\bf Weak Adversarial Networks (WANs).} 
      WANs solve inverse problems by employing the weak form of PDEs and training neural networks through adversarial learning\cite{ZANG2020109409}. In this approach, the neural network approximates the solution of the PDE by minimizing a loss function that incorporates both the weak formulation of the PDE and adversarial training techniques. This method allows the network to learn the underlying physical laws from observed data while effectively handling noise and incomplete information, providing a robust and flexible framework for solving complex inverse problems.
\end{itemize}

However, these methods typically require substantial internal measurement data to ensure that both the unknown physical information and the unknown field variables satisfy the PDEs within the solution domain. For entities that are difficult to penetrate, obtaining such internal data is often challenging. This limitation restricts the effectiveness and general applicability of AI methods in certain practical applications.

To further enhance the applicability of AI methods in solving inverse problems, we propose a novel neural network algorithm that relies solely on boundary information: Boundary-Informed Alone Neural Network (BIAN). The core idea of this approach is to transform the missing physical information in the governing equations into unknown equivalent excitation sources. By leveraging Green's theorem of energy conservation, an equation relating boundary energy flux to the energy distribution within the solution domain can be established\cite{10.1093/teamat/hrl017}. This method utilizes boundary information to predict the distribution of the equivalent excitation sources, thereby determining the unknown physical information and field distributions. The advantage of this approach is that it does not require any internal measurement data; it relies solely on boundary information to accurately predict and efficiently compute the unknown physical properties and field quantities. BIAN not only reduces the difficulty of data acquisition but also improves the feasibility of solving inverse problems, making it suitable for practical physical and engineering applications.

To enhance the robustness and accuracy of this algorithm, we propose a multi-neural network collaborative optimization model. This model includes three types of neural networks working together: a fully connected neural network (FCNN) as the approximator, another fully connected neural network as the generator, and a single-layer neural network as the discriminator. The FCNN approximates the missing physical information in the physical equations, the generator predicts the solution equation in the equations, and the discriminator evaluates the accuracy of the samples generated by the approximator on the boundary and inside the domain. Because the predictions of the generator and the discriminator can be expressed in closed form, the training process of these multiple neural networks allows for the exchange of information. This results in a collaborative rather than adversarial relationship among the networks, facilitating more efficient and accurate predictions.

An overview of this paper is as follows. \Sref{section2} is dedicated to our proposed approach for learning a neural network for the indeterminate coefficients. We describe topics including the mathematical framework, the structure of the network and the optimization methodology. In \Sref{section3}, we focus on the numerical examples that illustrate the benefits and potential of our approach in inverse problems with complex medium distributions. Finally, the conclusion and future work are provided in \Sref{section4}.

\section{The Algorithm}
\label{section2}

In this section, we delve into the mathematical underpinnings of the Boundary-Informed Alone Neural Network. This exploration encompasses three main components: the mathematical model of coefficient identification inverse problems, the mathematical principles underlying BIAN, and the network architecture utilized in this method.

\subsection{Mathematical Model of coefficient Identification Problem}

Firstly, we introduce the mathematical model for the coefficient identification inverse problems. These problems involve determining unknown coefficients of a system based on given system responses or observed data. Mathematically, such problems can be expressed as follows:
\begin{eqnarray}
\label{eq2_1}
           \mathcal{F}(u,\lambda) =f 
\end{eqnarray}
where $\mathcal{F}$ is the forward operator that describes the system's behavior. $u$ represents the state variable of the system, such as temperature, displacement, or electric potential, while $\lambda$ denotes the coefficients to be identified, such as material properties or source terms. The term $f$ represents the known system response or output. The goal of coefficient identification is to find the best estimate $\hat{\lambda}$ that minimizes the discrepancy between $\mathcal{F}(u,\hat{\lambda})$ and $f$.

However, we may address the following challenges when solving such coefficient identification inverse problems\cite{kirschIntroductionMathematicalTheory2021}:

\begin{itemize}
      \item {\bf Ill-posedness.} 
      One of the primary challenges in coefficient identification inverse problems is their ill-posed nature. According to Hadamard's definition, a problem is well-posed if a solution exists, the solution is unique, and it depends continuously on the input data. In contrast, ill-posed problems may not satisfy one or more of these criteria. For instance, there may be no solution, multiple solutions, or solutions that are highly sensitive to small perturbations in the input data. This sensitivity to measurement noise and errors can lead to significant inaccuracies in the identified coefficients, making it difficult to obtain reliable and stable solutions.
      \item {\bf Nonlinearity.}
      Another challenge is the nonlinearity inherent in the relationship between the state variable $u$, the coefficients $\lambda$, and the system response $f$. In many physical and engineering systems, the governing equations are nonlinear differential equations. This nonlinearity complicates the identification process because the solution space is not straightforward and may contain multiple local minima. Nonlinear optimization techniques are required to navigate this complex landscape, which can be computationally intensive and require sophisticated algorithms to ensure convergence to the global optimum.
      \item {\bf High Dimensionality.}
      coefficient identification problems can also involve a large number of coefficients, leading to high-dimensional optimization problems. High dimensionality increases the complexity of the problem significantly. The computational cost grows exponentially with the number of coefficients, a phenomenon known as the "curse of dimensionality." Moreover, high-dimensional spaces can be sparsely populated with data points, making it difficult to build accurate models and increasing the risk of overfitting. Efficient algorithms and dimensionality reduction techniques are essential to manage this complexity and make the optimization problem tractable.
\end{itemize}

\subsection{Mathematical Principles underlying BIAN}
In this section, we present the mathematical principles underlying our proposed method, BIAN. The core concept of BIAN is to leverage boundary information to accurately predict the internal distribution of physical properties and field quantities, without relying on internal measurement data. This is achieved by transforming the governing equations of the system into a form that can be solved using boundary data alone.

Firstly, we introduce how we establish the relationship between the energy distribution within the domain to the energy flux on the boundary, which eliminates the need for internal information for solving the PDEs. Contemplate a two-dimensional potential issue characterized by unknown medium distribution, structured as follows:
\numparts
\begin{eqnarray}
     \mathcal{L} u(x) = f(x), \quad x\in\Omega \label{eq2_2_1}\\
     u(x)=\bar{u}(x),\quad x\in\Gamma_1 \label{eq2_2_2}\\
     q(x)=\frac{\partial u(x)}{\partial n}=\bar{q}(x),\quad x\in \Gamma_2 \label{eq2_2_3}
\end{eqnarray}
\endnumparts
In \Eref{eq2_2_1}, $\mathcal{L}$ is the differential operator which has an indeterminate coefficient, $f(x)$ is the source element and $\Gamma_1+\Gamma_2=\partial \Omega$ is the boundary of the computational domain $\Omega$. \Eref{eq2_2_2} and \Eref{eq2_2_3} denote the essential boundary condition and the natural boundary condition for the potential issue. 

In this paper, we mainly force on the Poisson equations, i.e. $\mathcal{L}=-\nabla \cdot (\nabla \varepsilon(x))$, where $\varepsilon(x)$ represents the indeterminate medium distribution. The first step of this method involves transforming the partial differential equation with the form of \Eref{eq2_2_1} which contains indeterminate medium distributions into a Poisson equation in which all the indeterminate medium distributions are represented as the excitation source term. Upon elaborating on \Eref{eq2_2_1}, we obtain the following result:
\begin{eqnarray}
\label{eq2_3}
      -\varepsilon(x) \nabla^2 u(x) - \nabla \varepsilon(x) \nabla u(x) = f(x)
\end{eqnarray}
By rearranging the terms with indeterminate coefficients on the right side of the equation, we yield the function in the following form:
\begin{eqnarray}
\label{eq2_4}
      -\nabla^2 u(x) = g(x)
\end{eqnarray}
Here $g(x) = (f(x)+\nabla \varepsilon(x) \nabla u(x))/\varepsilon(x)$ donates the equivalent source term. The initial potential problem characterized by indeterminate medium distribution has been reformulated to encompass a potential problem with an unknown excitation source term. 

Then, we utilize the principle of energy conservation to establish the equation relating the energy distribution within the domain to the energy flux on the boundary. Prior to that, we introduce Green's function of PDEs, which plays the role of a powerful mathematical tool in solving PDEs. The Green's function, denoted as $G(x,x')$, represents the response of a system at a point $x$ due to a unit impulse source applied at another point $x'$. Mathematically, the Green's function satisfies the following fundamental equation:
\begin{eqnarray}
\label{eq2_5}
      \nabla^2 G(x,x') + \delta (x-x') = 0
\end{eqnarray}
where $\delta (x-x')$ is the Dirac delta function and $\nabla^2$ is the operator of the system. 

The Green's function has extensive applications in physics and engineering. For instance, in electromagnetics, it is used to determine the distribution of electric and magnetic fields; in acoustics, it helps analyze the propagation of sound waves; and in quantum mechanics, the Green's function is employed to study particle propagation and interactions.

In order to get the relationship between the energy distribution within the domain to the energy flux on the boundary, we propose the weighted residual approach by introducing the Green's function as the weight function. The weight residual function is represented as:
\begin{eqnarray}
\label{eq2_6}
      \int_{\Omega} G(x,x')(\nabla^2u(x)+g(x))\mathrm{d}\Omega=0
\end{eqnarray}
With integration by parts, we can get:
\begin{eqnarray}
\label{eq2_7}
      \int_{\Omega} G(x,x')(\nabla^2u(x)+g(x))\mathrm{d}\Omega=  \int_{\Omega}\frac{\partial}{\partial x_{i}}\Big(G(x,x') \frac{\partial u(x)}{\partial x_{i}}\Big)\mathrm{d}\Omega \nonumber \\-\int_{\Omega}\frac{\partial u(x)}{\partial x_{i}}\frac{\partial G(x,x')}{\partial x_{i}}\mathrm{d}\Omega + \int_{\Omega} G(x,x')g(x)\mathrm{d}\Omega
\end{eqnarray}
Performing integration by parts again on the second integral term on the right-hand-side of \Eref{eq2_7}, we obtain:
\begin{eqnarray}
\label{eq2_8}
      \int_{\Omega}\frac{\partial u(x)}{\partial x_{i}}\frac{\partial G(x,x')}{\partial x_{i}}\mathrm{d}\Omega \nonumber \\ = \int_{\Omega}\frac{\partial}{\partial x_{i}}\Big( u(x)\frac{\partial G(x,x')}{\partial x_{i}}\Big)\mathrm{d}\Omega-\int_{\Omega} u(x)\nabla^2G(x,x')\mathrm{d}\Omega 
\end{eqnarray}

Gauss's theorem, with the form of \Eref{eq2_9}, has been applied to establish the equation relating the energy distribution within the domain to the energy flux on the boundary. 
\begin{eqnarray}
\label{eq2_9}
      \int_{\Omega}(\nabla F)\mathrm{d}\Omega = \oint _{\partial \Omega}F\cdot \mathrm{d}S
\end{eqnarray}
The relationship between the two kinds of energy can be obtained as:
\begin{eqnarray}
\label{eq2_10}
      c(P)u(P)+\int_{\Gamma}u(Q)\frac{\partial u^{*}(Q,P)}{\partial\boldsymbol{n}}\mathrm{d}\Gamma(Q)= \nonumber \\ \int_{\Gamma}u^{*}(Q,P)\frac{\partial u(Q)}{\partial\boldsymbol{n}}\mathrm{d}\Gamma(Q)+\int_{\Omega}u^{*}(q,P)g(q)\mathrm{d}\Omega(q)
\end{eqnarray}
Here $Q \in \partial \Omega$ and $q\in \Omega$ are the field points on the boundary and inside the domain, respectively. $P \in \partial \Omega$ represents the boundary point of the domain acting as the source point. $c(P)$ is dictated by the location of point $P$. Specifically, $c(P) = 1 - a/2\pi$ when point $P$ is positioned on the boundary of the domain, where $a$ is the angle constituted by the tangent plane at the boundary point $P$.  

When the source point $p$ is located inside the domain $\Omega$, according to the properties of the Dirac delta function, we can get the relationship between the internal points and the boundary condition with the form:
\begin{eqnarray}
\label{eq2_11}
      u(p)+\int_{\Gamma}u(Q)\frac{\partial u^{*}(Q,p)}{\partial\boldsymbol{n}}\mathrm{d}\Gamma(Q)= \nonumber \\ \int_{\Gamma}u^{*}(Q,p)\frac{\partial u(Q)}{\partial\boldsymbol{n}}\mathrm{d}\Gamma(Q)+\int_{\Omega}u^{*}(q,p)g(q)\mathrm{d}\Omega(q) 
\end{eqnarray}
Here $p\in \Omega$ represents the internal points, acting as the source points.

Through the above method, we represent the energy distribution within the solution domain using the boundary energy flux, resulting in an energy equation that relies solely on boundary information. Next, we introduce how to use neural networks to approximate the medium parameters and potential functions.

The method of utilizing multiple neural networks in combination to enhance solution accuracy and speed has seen extensive development, including techniques such as Generative Adversarial Networks (GANs)\cite{goodfellow2014generativeadversarialnetworks}. However, these methods typically involve an adversarial relationship between the neural networks. In contrast, BIAN employs a collaborative training approach for neural networks. Collaborative training involves the simultaneous training of two or more neural networks, allowing them to interact and learn from each other during the training process. This interaction through shared information and gradients can lead to better performance compared to training a single network independently.

In this work, we propose a collaborative optimization model involving three neural networks to address the coefficient identification problems. These networks work synergistically, sharing information to enhance solution accuracy and computational efficiency. This collaborative approach provides a robust framework for solving complex physical equations. We will propose each of the three neural networks individually, detailing the structures, the roles they play and the mathematical expressions.

The first neural network plays the role of an approximator to estimate the indeterminate medium distributions of the physical equation. In the architecture that we use, apart from the input and output layers, two residual blocks are used consisting of two linear transformations, two activation functions and a residual connection, both the input $x_{i-1}$ and the output $x_i$ of the residual block are vectors in $\mathbb{R}^m$. The residual block of the neural network can be expressed as:
\begin{eqnarray}
\label{eq2_12}
      x_{i} = f^{i-1}_{\vartheta}(x_{i-1};\vartheta ) = \sigma (W_{i-1}^2 \cdot \sigma (W_{i-1}^1 x_{i-1}+b_{i-1}^1) + b_{i-1}^2) + x_{i-1}
\end{eqnarray}
Where $W_{i-1}^1, W_{i-1}^2 \in \mathbb{R}^{m\times m}$, $b_{i-1}^1, b_{i-1}^2 \in \mathbb{R}^{m}$ are parameters associated with the block. $\sigma$ is the activation function. To balance the simplicity and accuracy, we decide to use the activation function with the form:
\begin{eqnarray}
\label{eq2_13}
      \sigma(x) = max\{ x^3,0\} 
\end{eqnarray}

The last term of \Eref{eq2_12}, the residual connection, is a mature architecture in deep learning to enhance the performance of the model. The importation of the residual connection helps to address the vanishing gradient problem and the degradation problem.\cite{Hu_2024} The structure of this neural network is shown as \Fref{fig2_1}:
\begin{figure}[H]
      \centering
      \includegraphics[scale=0.7]{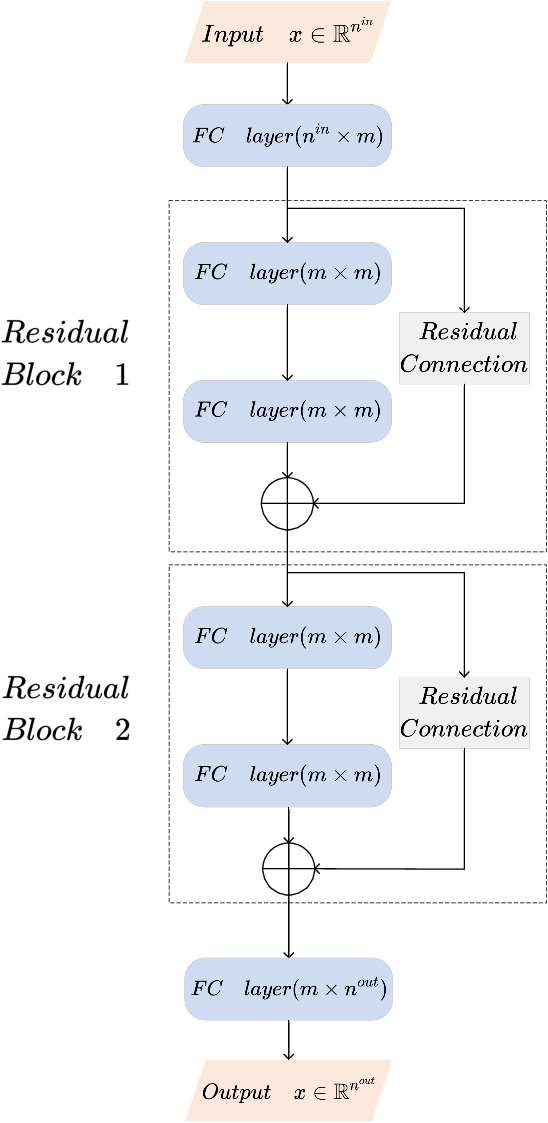}
      \caption{The structure of the neural network used in this work. Two residual blocks and two extra fully connected layers are employed.}
      \label{fig2_1}
\end{figure}

The neural network can be expressed as:
\begin{eqnarray}
\label{eq2_14}
      u_1(x;\theta_1) = W_{output} \cdot( f^2_{\vartheta} \circ f^1_{\vartheta}(\sigma (W_{input} \cdot x +b_{input}))) + b_{output}
\end{eqnarray}
Where $\theta_1$ donates the full parameters of the neural network, including the weight and bias in the residual blocks and the input and output layers.

As this neural network plays the role of approximator of the indeterminate medium distributions, how to train such the network to learn the unknown physical information is the key problem. Taking \Eref{eq2_10} into consideration, the indeterminate medium distribution $\varepsilon(x)$ is entirely determined by the source term $g(x)$, which can be represented by essential boundary condition and natural boundary condition solely. Compared with the internal information, the boundary condition is easier to obtain. 

Substituting the essential boundary condition and the natural boundary condition with \Eref{eq2_10}, we can get the residual for a given boundary source point $P$:
\begin{eqnarray}
\label{eq2_15}
      &R_1(P)=c(P)u(P)+\int_\Gamma u(Q)\frac{\partial u^*(Q,P)}{\partial\boldsymbol{n}}d\Gamma(Q) \nonumber \\&-\int_\Gamma u^*(Q,P)\frac{\partial u(Q)}{\partial\boldsymbol{n}}d\Gamma(Q)-\int_\Omega u^*(q,P)g(q)d\Omega(q)
\end{eqnarray}

Utilizing the neural network $u_1(x;\theta_1)$ to replace the source term $g(x)$, we can define the loss function that guides the training of the neural network to learn the indeterminate medium distribution. Suppose $N_b$ source points are allocated on the boundary. Then the loss function is taken as the average of the residuals among the source points $\{P_i\}$:
\begin{eqnarray}
\label{eq2_16}
      Loss_1(\theta_1)=\frac{1}{N_b}\sum_{i=1}^{N_b}||R_1(P_i,\theta_1)||^2
\end{eqnarray}

The equivalent source term can be obtained by minimizing the loss function through the training process:
\begin{eqnarray}
\label{eq2_17}
      \theta_1^*=\arg\min_{\theta\in\Theta_1}Loss_1(\theta_1)     
\end{eqnarray}

However, employing a single neural network does not directly yield the solution function for the coefficient identification problems we aim to solve. To address this, we introduce a second neural network as a generator. Similar to the approximator, we utilize a neural network with two residual blocks, which has the same form as the approximator.

By simultaneously utilizing the source term distribution obtained from the first neural network and integrating it into \Eref{eq2_11}, we can generate the solution function of the partial differential equation. The residual for a given internal source point $p$ is expressed as:
\begin{eqnarray}
\label{eq2_18}
      R_{2}\left(p, \theta_{2}\right)=u\left(p\right)+\int_{\Gamma} u\left(Q\right) \frac{\partial u^{*}(Q, p)}{\partial \boldsymbol{n}} d \Gamma(Q) \nonumber \\-\int_{\Gamma} u^{*}(Q, p) \frac{\partial u(Q)}{\partial \boldsymbol{n}} d \Gamma(Q)-\int_{\Omega} u^{*}(q, p) u_{1}\left(q, \theta_{1}\right) d \Omega(q)
\end{eqnarray}

Except for the first term on the right-hand side of \Eref{eq2_18}, the other terms can be obtained through the two types of boundary conditions and the trained first neural network. Utilizing the generating neural network $u_2(x;\theta_2)$ to replace the solution $u(p)$ of the internal points $p$, we can define the loss function that guides the training of the neural network to learn the solution equation. Suppose $N_i$ source points are allocated on the boundary. Then the loss function is taken as the average of the residuals among the source points $\{p_i\}$:
\begin{eqnarray}
      \label{eq2_19}
            Loss_2(\theta_2)=\frac{1}{N_i}\sum_{i=1}^{N_i}||R_2(p_i,\theta_2)||^2
\end{eqnarray}

By training the second neural network to minimize this loss function, we derive the solution function for the PDE.The equivalent source term $g(x)$ and the solution function $u(x)$ can be obtained through the training process of the two neural networks. 

Given the established relationship between the equivalent source term and the indeterminate medium coefficients, as expressed in \Eref{eq2_20}, we can infer the indeterminate medium coefficients from the derived equivalent source term:
\begin{eqnarray}
      \label{eq2_20}
      g(x) = (f(x)+\nabla \varepsilon(x) \nabla u(x))/\varepsilon(x)
\end{eqnarray}

This collaborative approach between the two neural networks ensures an accurate and efficient solution to the coefficient identification problem, adhering to the physical constraints and boundary conditions.

To further ensure the accuracy of the computed results, we introduce a third neural network as a discriminator. This network is designed to evaluate the precision of the approximations generated by the first two neural networks.

The discriminator neural network's primary function is to assess the accuracy of the generated solutions by comparing them with known boundary conditions and internal points. This evaluation helps in refining the predictions made by the approximator and generator networks, ensuring higher accuracy and consistency.

The loss function for the discriminator neural network is designed to minimize the difference between the predicted solutions and the actual boundary conditions. Suppose $N_d$ points are allocated on the boundary and within the domain for evaluation. The loss function is given by:
\begin{eqnarray}
      \label{eq2_21}
      Loss_3(\theta_2)=\frac{1}{N_d}\sum_{i=1}^{N_d}||u_2(p_i;\theta)-u_{actual}(p_i)||^2
\end{eqnarray}
\begin{figure}[H]
      \centering
      \includegraphics[scale=0.7]{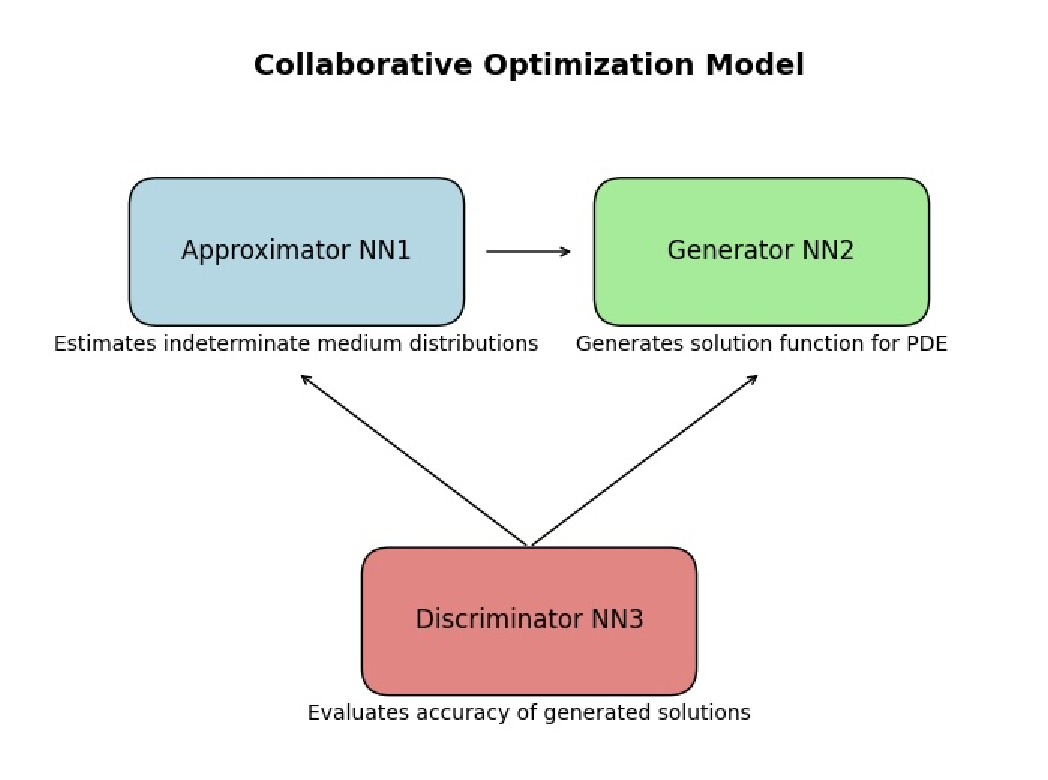}
      \caption{The relationships between the three neural networks.}
      \label{fig2_2}
\end{figure}

With the introduction of the discriminator, our collaborative optimization model now includes three neural networks working together:
\begin{itemize}
      \item {\bf approximator.} 
      Estimates the indeterminate medium distribution.
      \item {\bf Generator.}
      Generates the solution function for the PDE.
      \item {\bf Discriminator.}
      Evaluates the accuracy of the generated solutions.
\end{itemize}

This collaborative framework enhances the overall accuracy and robustness. The approximator and generator networks work together to solve the PDE, while the discriminator ensures the solutions are precise and consistent with known boundary and internal conditions. The relationships between the three neural networks are shown as \Fref{fig2_2}.

In summary, we conclude the algorithm of BIAN in Algorithm \eref{algorithm_1}.

In this section, we presented a novel collaborative optimization model involving three neural networks, the approximator, the generator and the discriminator, to solve the coefficient identification problems in PDEs. What's more, with the principle of energy conservation, we establish the equation relating the energy distribution within the domain to the energy flux on the boundary, which leverages boundary information to accurately predict the internal distribution of physical properties and field quantities. The proposed method addresses the challenges inherent in inverse problems, such as ill-posedness, nonlinearity, and high dimensionality, making it a powerful tool for solving a wide range of physical and engineering problems. The collaborative optimization model demonstrates significant potential for practical applications, providing a robust framework for accurately identifying unknown coefficients in PDEs. 

In the next section, we will demonstrate the convergence properties of the proposed algorithm. Compared with PINN, this algorithm exhibits a significant advantage in terms of convergence speed. 

\begin{algorithm}[H]
      \renewcommand{\algorithmicrequire}{\textbf{Input:}}
        \renewcommand{\algorithmicensure}{\textbf{Output:}}
        \caption{Algorithm of BIAN}
      \label{algorithm_1}
      \begin{algorithmic}[1] 
            \REQUIRE  Allocate $N_b$ source points $\{P_i\}$ and $M_b$ integration points $\{ Q_i \}$ on the boundary, $N_i$ source points $\{p_j\}$ and $M_i$ integration points $\{ q_j \}$ inside the domain.
            \ENSURE The distribution of the solution function and the medium in the solution domain.
          
            \STATE For each $P_i$ on the boundary, calculate the kernel function on all the internal integration points $\{q_j\}$.

            \STATE For each $P_i$ on the boundary, calculate the kernel function on all the boundary integration points $\{Q_i\}$.

            \STATE For each $p_j$ inside the domain, calculate the kernel function on all the internal integration points $\{q_j\}$.

            \STATE For each $p_i$ on the boundary, calculate the kernel function on all the boundary integration points $\{Q_i\}$.

            \STATE Initialize the parameters of three networks $\theta_1$, $\theta_2$ and $\theta_3$. Choose the number of iteration steps $N$ for the networks, respectively. 
            \FOR {$i=1,2,\cdots,N$}
                  \STATE Calculate the residual shown as \Eref{eq2_15} for all the source points $\{P_i\}$, then compute the loss function for the approximator following \Eref{eq2_16}. Update the parameters $\theta_1$ with the optimization algorithm.
                  \STATE Calculate the residual shown as \Eref{eq2_18} for all the source points $\{p_i\}$ inside the domain, then compute the loss function for the generator following \Eref{eq2_19} with the result of the first neural network. Update the parameters $\theta_2$ with the optimization algorithm.
                  \STATE Evaluate the accuracy of the outputs of the two networks by comparing them with given boundary and internal points, then compute the loss function for the discriminator with the form as \Eref{eq2_21}. Update the parameters $\theta_3$ with the optimization algorithm.
                  \STATE Refine the first two networks utilizing the feedback from the discriminator network.
            \ENDFOR
            \STATE After the training process, both the medium distributions and the resolution are solved.
      \end{algorithmic}
\end{algorithm}

\section{The convergence analysis of BIAN}
\label{section5}
In this section, we provide a rigorous convergence theory for BIAN with respect to the amount of training data. By employing probabilistic space-filling parameters, we derive an upper bound on the expected unregularized BIAN loss. Specifically, we utilize a probabilistic framework to analyze the learning dynamics of BIAN, establishing that as the number of training data points increases, the expected loss decreases with high probability. This bound offers a theoretical guarantee that, under certain regularity conditions, BIAN will converge to an accurate solution as the training set grows. Furthermore, we demonstrate that the convergence rate depends on the dimensionality of the problem and the smoothness of the underlying solution, providing insight into the trade-offs between data complexity and model performance. These results lay the foundation for understanding the performance of BIAN in practical applications and offer a clear direction for optimizing its training process to enhance efficiency and accuracy.

Before the demonstration, we introduce the mathematical principle which helps the demonstration of the convergence theorem. 
\begin{definition}[Hölder Continuity]
      Let \( f: \mathbb{R}^n \to \mathbb{R} \) be a function, and let \( \a lpha \in (0,1] \) and \( C > 0 \). We say that \( f \) is \textit{Hölder continuous} with exponent \( \alpha \) if for all \( x, y \in \mathbb{R}^n \), the following inequality holds:
      \[
      |f(x) - f(y)| \leq C |x - y|^\alpha.
      \]
      Here, \( \alpha \) is called the Hölder exponent, and \( C \) is a Hölder constant. When \( \alpha = 1 \), the function \( f \) is said to be \textit{Lipschitz continuous}.
\end{definition}

We call $[f]_{a;U}$, with the form of \Eref{eq4_1}, as the Hölder constant(coefficient) of $f$ on $U$.
\begin{eqnarray}
\label{eq4_1}
      [f]_{a;U}=\sup_{x,y\in U,x\neq y}\frac{|f(x)-f(y)|}{\|x-y\|^{\alpha}}<\infty,\quad0<\alpha\leq1,
\end{eqnarray}

As BIAN is aim to deal with problems in which indeterminate parameters in the differential operator with a set of training data. The convergence rate of the first Neural Network, which serves as the approximator, plays a crucial role in determining the overall convergence speed of the algorithm. Therefore, our primary focus is on analyzing the convergence properties of the approximator neural network. The training data consist of two types of data sets: boundary and internal integral data. A boundary integral datum is composed of the coordinate information and the two kinds of boundary conditions $(x_b,u(x_b),\partial u(x_b)/\partial \boldsymbol{n})$, where $x_b \in \partial U$, and a internal integral datum just contains the coordinate information $(x_i)$, where $x_i \in U$. The set of $m_b$ boundary integral data and the set of $m_i$ internal integral data are donated by $\mathcal{T}_{b}^{m_{b}}=\{\mathbf{x}_{b}^{i}\}_{i=1}^{m_{b}}$ and $\mathcal{T}_{i}^{m_{i}}=\{\mathbf{x}_{i}^{j}\}_{j=1}^{m_{i}}$, respectively. To approximate the solution of the PDE with the form as \Eref{eq2_2_1}, \Eref{eq2_2_2}, and \Eref{eq2_2_3}, we seek to find a neural network $h^* \in \mathcal{H}_{n}$ to minimize the loss function with the form as \Eref{eq4_2}.
\begin{eqnarray}
\label{eq4_2}
      L(x_{b},x_{i};h)=||F(x_{b})-A[h](x_i)||^{2}
\end{eqnarray}
Here $F(x_{b}) = m(x_b) + \int_\Gamma n(x_b)d\Gamma = c(x_b)u(x_b) + \int_\Gamma u(x_1)\frac{\partial u^*(x_1,x_b)}{\partial\boldsymbol{n}}-u^*(x_1,x_b)\frac{\partial u(x_1)}{\partial\boldsymbol{n}}d\Gamma(x_1)$, $A[h](x_i) =\int_\Omega z(x_i)d\Omega =\int_\Omega u^*(x_2,x_i)h(x_i)d\Omega(x_2)$. $x_1, x_2$ represent the points on the boundary and internal, respectively. 

Support $\mathcal{T}_{b}^{m_{b}}$ and $\mathcal{T}_{i}^{m_{i}}$ areindependently and identically distributed (iid) samples from probability distributions $\mu_b$ and $\mu_i$. Since the empirical probability distribution on $\mathcal{T}_{b}^{m_{b}}$ defined by $\mu_b^{m_b}=\frac{1}{m_b}\sum_{i=1}^{m_b}\delta_{x_b^i}$, the $\mathcal{T}_{i}^{m_{i}}$ is difined Similarly, the empirical loss function and expected loss function are obtained by taking the expectation on \Eref{eq4_2} with respect to $\mu^m=\mu_b^{m_b}\times\mathcal{T}_{i}^{m_{i}}$ and $\mu=\mu_b\times\mu_i$, respectively.
\begin{eqnarray}
\label{eq4_3}
      \mathrm{Loss}_{\boldsymbol{m}}(h)=\mathbb{E}_{\mu^{m}}[\mathbf{L}(\mathbf{x}_{b},\mathbf{x}_{i};h)], \quad\mathrm{Loss}(h)=\mathbb{E}_{\mu}[\mathbf{L}(\mathbf{x}_{b},\mathbf{x}_{i};h)].
\end{eqnarray}
If the expected loss function were available, the minimizer of this function would be the solution of the PDE or close to it. However, it is unavailable to get it in practice, the empirical loss function is used as a substitute. We will give an upper bound of the expected loss function, which involves the empirical loss function. The derivation is based on the probabilistic space filling arguments. In this regard, we make the following assumptions on the training data distributions.
\begin{assumption}
\label{as_1}
Let $U$ be a bounded domain in ${R}^{d}$ that is at least of class $C^{0,1}$ and $\Gamma$ be a closed subset of $\partial U$. Let $\mu_b$ and $\mu_i$ be probability distributions defined on $\Gamma$ and $U$, respectively. Let $\rho_b$ be the probability density of $\mu_b$ with respect to $d-1$-dimensional Lebesgue measure on $\Gamma$. Let $\rho_i$ be the probability density of $\mu_i$ with respect to $d$-dimensional Hausdorff measure on $U$.
\begin{enumerate}
      \item[1.] \( \rho_r \) and \( \rho_b \) are supported on \( \overline{U} \) and \( \Gamma \), respectively. Also, \( \inf_{\overline{U}} \rho_r > 0 \) and \( \inf_{\Gamma} \rho_b > 0 \).
      
      \item[2.] For \( \varepsilon > 0 \), there exists partitions of \( U \) and \( \Gamma \), \( \{ U_j^{\varepsilon} \}_{j=1}^{K_r} \) and \( \{ \Gamma_j^{\varepsilon} \}_{j=1}^{K_b} \), that depend on \( \varepsilon \), such that for each \( j \), there are cubes \( H_{\varepsilon}(z_{j,r}) \) and \( H_{\varepsilon}(z_{j,b}) \) of side length \( \varepsilon \) centered at \( z_{j,r} \in U_j^{\varepsilon} \) and \( z_{j,b} \in \Gamma_j^{\varepsilon} \), respectively, satisfying \( U_j^{\varepsilon} \subset H_{\varepsilon}(z_{j,r}) \) and \( \Gamma_j^{\varepsilon} \subset H_{\varepsilon}(z_{j,b}) \).
      \item[3.] For each \( m \), \( \mathcal{H}_m \) contains a network \( u_m^* \) satisfying 
      \(
      Loss_m(h_m^*) = 0.
      \)

      \item[4.] There exist positive constants \( c_r, c_b \) such that \( \forall \varepsilon > 0 \), the partitions from the above satisfy \( c_r \varepsilon^d \leq \mu_r(U_j^{\varepsilon}) \) and \( c_b \varepsilon^{d-1} \leq \mu_b(\Gamma_j^{\varepsilon}) \) for all \( j \).
      
      There exist positive constants \( C_r, C_b \) such that \( \forall x_r \in U \) and \( \forall x_b \in \Gamma \), \( \mu_r(B_{\varepsilon}(x_r) \cap U) \leq C_r \varepsilon^d \) and \( \mu_b(B_{\varepsilon}(x_b) \cap \Gamma) \leq C_b \varepsilon^{d-1} \), where \( B_{\varepsilon}(x) \) is a closed ball of radius \( \varepsilon \) centered at \( x \).
      
      Here \( C_r, c_r \) depend only on \( (U, \mu_r) \) and \( C_b, c_b \) depend only on \( (\Gamma, \mu_b) \).
  \end{enumerate}
  
\end{assumption}

We now state our result that bounds the expected PINN loss in terms of the empirical loss. Let us recall that $m$ is the vector of the number of training data points, i.e., $m = (m_b,m_i)$. The constants $c_b,C_b,c_i,C_i$ are introduced in Assumption 3.1. For a function $f$, $[f]_{a;U}$ is the Hölder constant of $f$ with exponent $a$ in $U$.

\begin{theorem}
      \label{th_1}
      Suppose Assumption 3.1 holds. Let $m_b$ and $m_i$ be the number of iid samples from $\mu_b$ and $\mu_i$, respectively. With probability at least, $(1-\sqrt{m_{b}}(1-1/\sqrt{m_{b}})^{m_{b}})(1-\sqrt{m_{i}}(1-1/\sqrt{m_{i}})^{m_{i}})$, we have
      \begin{eqnarray}
      \label{eq4_4}
      \mathrm{Loss}(h)\leq C'_{1}\cdot \mathrm{Loss}_{\boldsymbol{m}}(h)+C'_{max}(m_b^{-\frac{a}{d-1}-0.5}+m_i^{-\frac{a}{d}-0.5}),
      \end{eqnarray}
      here $C'_1=3 \frac{C_{b}C_{i}}{c_b c_i}m_b^{0.5} m_i^{0.5} \sqrt d ^{2d-1}$ and $C_{max} = max\{ 2[n]^2_{a,U}d^{a+0.5d-0.5} c_b^{-\frac{2a+d-1}{d-1}}\\, 2[z]^2_{a,U} d^{a+0.5d} c_b^{-\frac{2a+d}{d}}\}$.
\end{theorem}
\begin{proof}
      The proof can be found in Appendix A.
\end{proof}

As the number of boundary point $m_b$ and the internal point $m_i$ has the relationship as:
\begin{eqnarray}
      m_b = O(m_i^{\frac{d-1}{d}})
\end{eqnarray}
The number of training data is depends on the number of boundary point. We can obtain
\begin{eqnarray}
\label{eq4_5}
Loss(h)=O(m_b^{-\frac{a}{d-1}-0.5}).
\end{eqnarray}

From Theorem \ref{th_1}, we establish the relationship between the convergence rate of BIAN with the number of boundary data number. In the next section, we provide a numerical experiment to illustrate our theoretical findings. Also, experiments to prove the feasibility and accuracy of BIAN have been conducted. 

\section{Numerical Examples}
\label{section3}
In this section, we focus on the convergence and the feasibility of the Boundary-Informed Alone Neural Network (BIAN), as well as the impact of collaborative training on solution efficiency. To evaluate these aspects, we conducted a series of numerical experiments. These experiments were designed to address two-dimensional potential problems with varying complexity, including scenarios with indeterminate spatially varying medium distributions, indeterminate piecewise uniform medium distributions, as well as the convergence rates of different methods applied to the same problem.

The experiments aimed to demonstrate the robustness and precision of BIAN in accurately approximating unknown medium parameters and solving the associated partial differential equations. By comparing the performance of BIAN with existing methods, we highlighted its superior computational accuracy and efficiency. Furthermore, we assessed the contribution of the collaborative training framework, involving the approximator, generator, and discriminator neural networks, in enhancing the overall solution process.

The results from these experiments provide comprehensive insights into the effectiveness of BIAN in handling complex inverse problems, showcasing its potential for practical applications in various fields such as engineering and physics. The following sections detail the experimental setup, methodologies, and findings, underscoring the advantages of the proposed approach over traditional techniques.

\subsection{Evaluation of Computational Feasibility}
\label{sec4_1}
In this section, we examine the computational feasibility of BIAN by analyzing its performance in solving complex inverse problems. The primary focus is on evaluating the efficiency of the collaborative optimization model and its ability to handle high-dimensional data and complex medium distributions. We conducted a numerical experiment with a two-dimensional Laplace problem to benchmark BIAN against traditional methods, including Physics-Informed Neural Networks (PINN), Deep Ritz Method (DRM), and Weak Adversarial Networks (WAN), assessing key metrics such as computational time, resource usage, and scalability. The results demonstrate BIAN's capability to provide accurate solutions within reasonable computational limits, highlighting its potential for practical applications in engineering and physical sciences.

Consider the Laplace equation with indeterminate spatially varying medium distribution:
\numparts
\begin{eqnarray}
     -\nabla \cdot (\varepsilon(\mathbf{x})\nabla u(\mathbf{x}) = f(\mathbf{x}), \quad &\mathbf{x}\in\Omega \label{eq3_1_1}\\
     u(\mathbf{x})=sin(\pi y),\quad &\mathbf{x}\in\Gamma_1 \label{eq3_1_2}\\
     u(\mathbf{x})=0,\quad &\mathbf{x}\in \Gamma_2 \label{eq3_1_3}
\end{eqnarray}
\endnumparts
Here $\Omega=\{(x,y)|(x,y)\in(0,1)\times(0,1)\}$, $\Gamma_1 = \{ (x,y)|(x,y)=\{0,1\}\times (0,1) \}$, $\Gamma_2 = \partial \Omega \backslash \Gamma_1$, $\varepsilon(\mathbf{x})=e^{\pi^2(x-x^2)/2} / \pi$ and $f(x)=0$, the solution of this problem $u(\mathbf{x})=e^{\pi^2(x^2-x)/2}\cdot sin(\pi y)$. The natural boundary condition can be obtained from the analytical expression of the solution. The solution and the medium distribution of the Laplace problem is shown as \Fref{fig3_1}, and we can get the analytical solution of the natural boundary condition.
\begin{figure}[H]
	\centering
	\begin{minipage}{0.49\linewidth}
		\centering
		\includegraphics[width=1\linewidth]{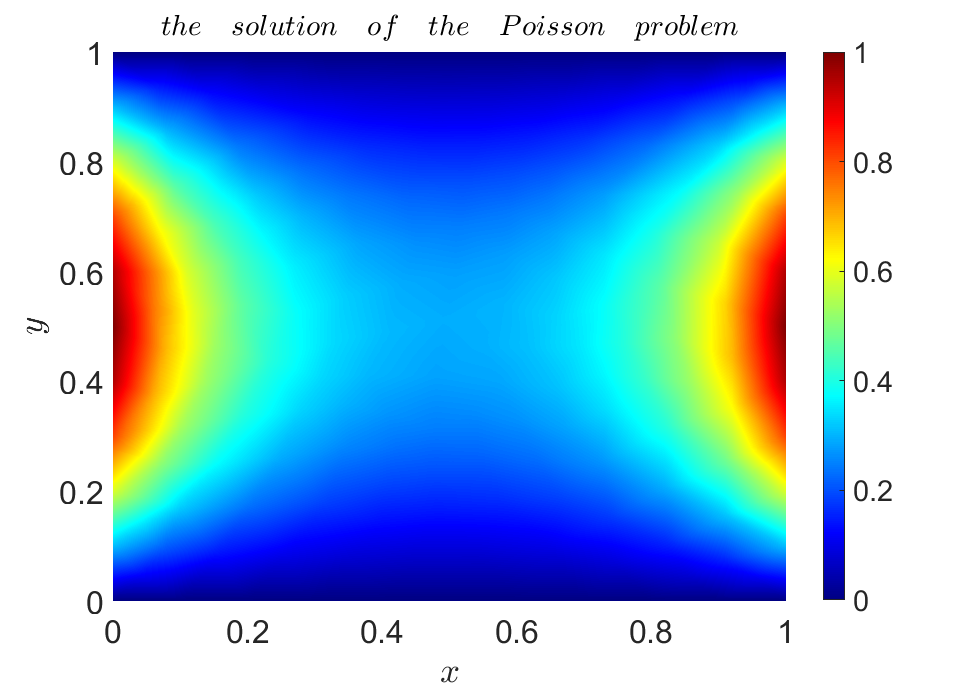}
		\label{fig3_1}
	\end{minipage}
	\begin{minipage}{0.49\linewidth}
		\centering
		\includegraphics[width=1\linewidth]{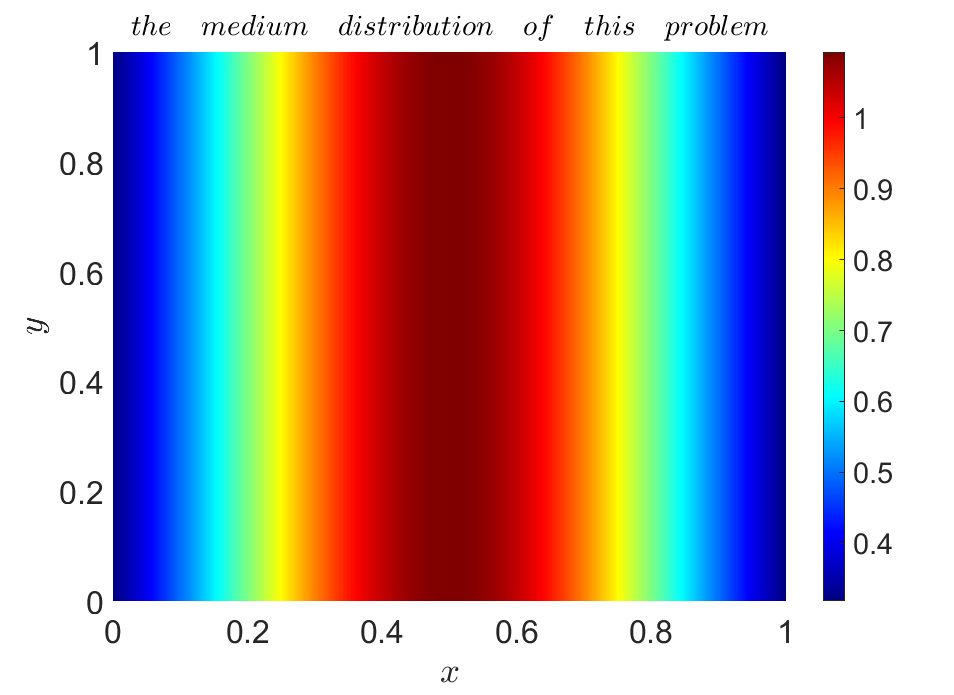}
		\label{fig3_2}
	\end{minipage}
      \caption{The solution and the medium distribution of the Laplace problem.}
\end{figure}

Each layer of the residual blocks in the approximator network and generator network we used to solve this problem has $m=10$ neurons, there are a total of parameters in this model. We select 10 evenly spaced points on each boundary as the boundary source points $\{ P_i\}$ and 100 random points inside the domain as the internal source point $\{ p_j\}$. The selection of the integration points is consistent with the source points. Since we want to compare this method with existing methods, such as PINN, DRM and WAN, the three methods are also utilized to solve this Laplace problem. For the three methods, we use a multilayer perceptron (MLP) with 4 hidden layers and each layer has 20 neurons and we choose 500 random points as the training points. 

The solution and the medium distribution obtained from each method is shown as follows and the $L_2$ error of each method is shown in \Tref{tbl_1}. By comparing the results from different methods, it is evident that BIAN not only demonstrates better accuracy but also achieves faster convergence compared to the other three proposed methods. More importantly, BIAN requires only boundary information to solve the problem, significantly simplifying the data acquisition process.
\begin{table}[h]
      \centering
      \caption{The $L_2$ error of the solution and medium distribution for the four methods} \label{tbl_1}
      \begin{tabular}{l|llll}
          \hline
          \diagbox{$L_2$ error}{Method}& PINN  & WAN& DRM & BIAN \\
          \hline
          medium distribution & 0.0313 & 0.0526 & 0.9872 &0.0143\\
          \hline
          solution & 0.3117 & 0.2978 & 0.2281& 0.0121 \\
          \hline
      \end{tabular}
\end{table}

\begin{figure}[H]
      \centering
      \includegraphics[scale=0.49]{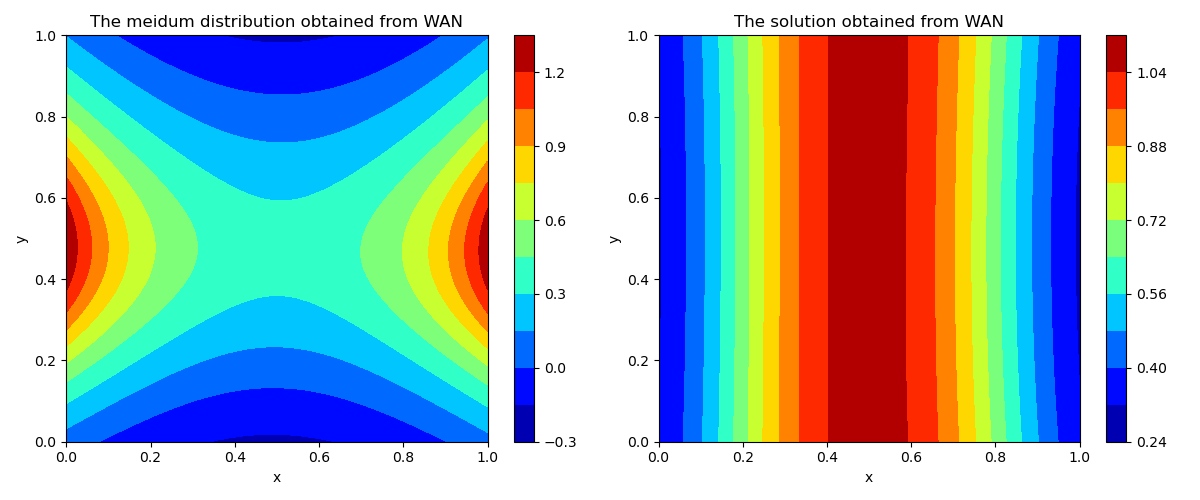}
      \caption{The solution and the medium distribution of the Laplace problem obtained from WAN.}
      \label{fig3_3}
\end{figure}
\begin{figure}[H]
      \centering
      \includegraphics[scale=0.49]{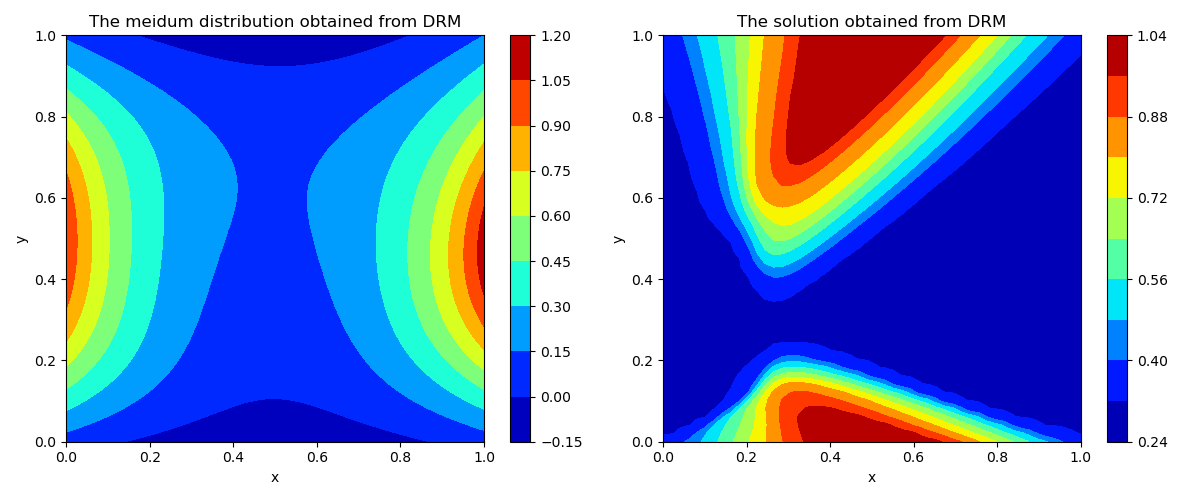}
      \caption{The solution and the medium distribution of the Laplace problem obtained from DRM.}
      \label{fig3_4}
\end{figure}
\begin{figure}[H]
      \centering
      \includegraphics[scale=0.49]{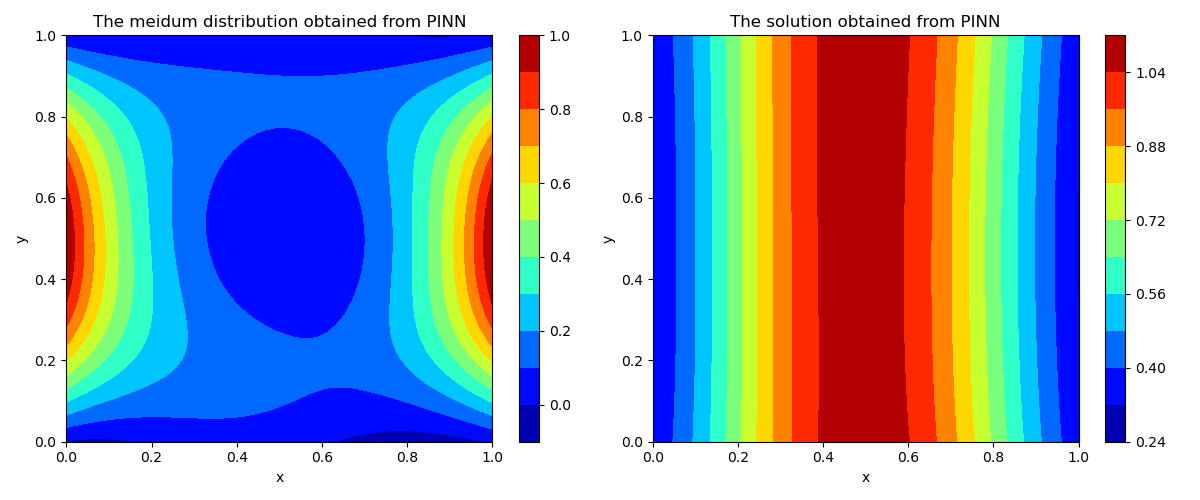}
      \caption{The solution and the medium distribution of the Laplace problem obtained from PINN.}
      \label{fig3_5}
\end{figure}
\begin{figure}[H]
	\centering
	\begin{minipage}{0.49\linewidth}
		\centering
		\includegraphics[width=1\linewidth]{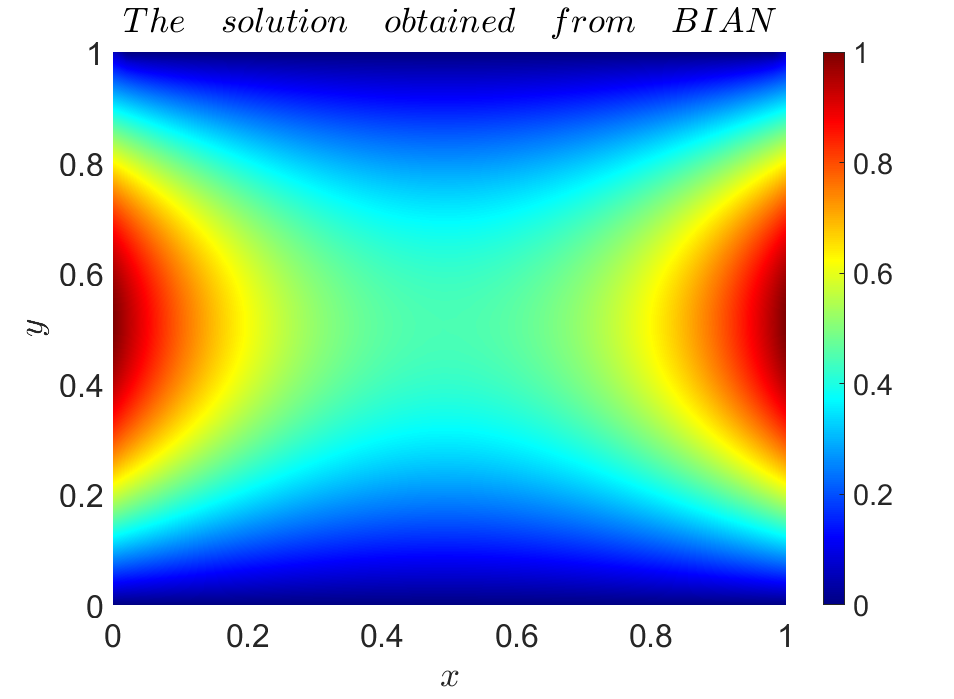}
		\label{fig3_6}
	\end{minipage}
	\begin{minipage}{0.49\linewidth}
		\centering
		\includegraphics[width=1\linewidth]{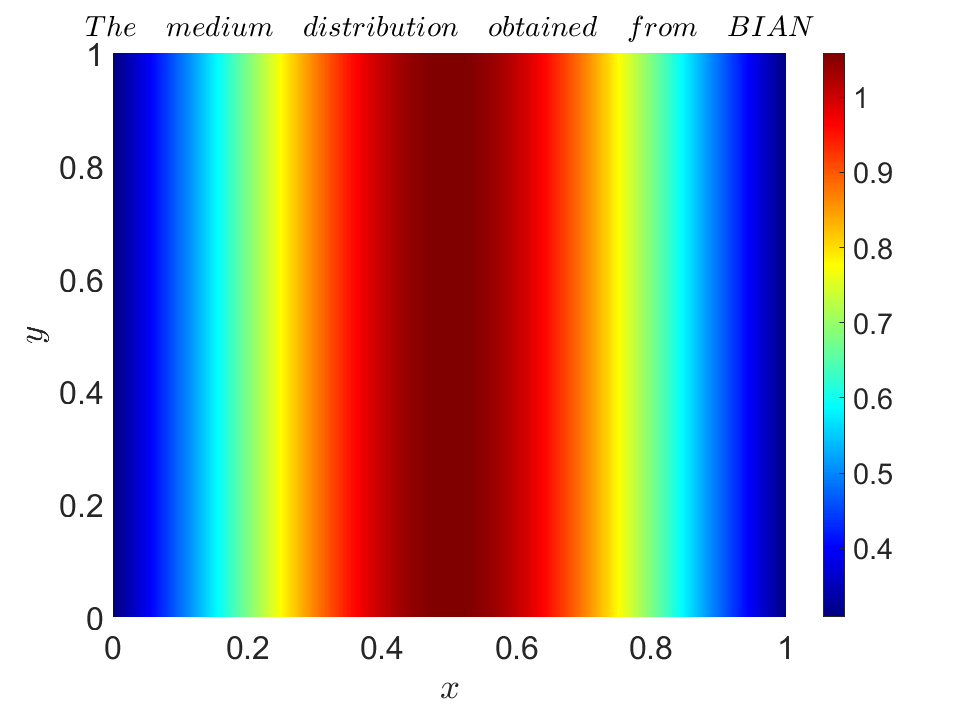}
		\label{fig3_7}
	\end{minipage}
      \caption{The solution and the medium distribution of the Laplace problem obtained from BIAN.}
\end{figure}
\begin{figure}[H]
      \centering
      \includegraphics[scale=0.49]{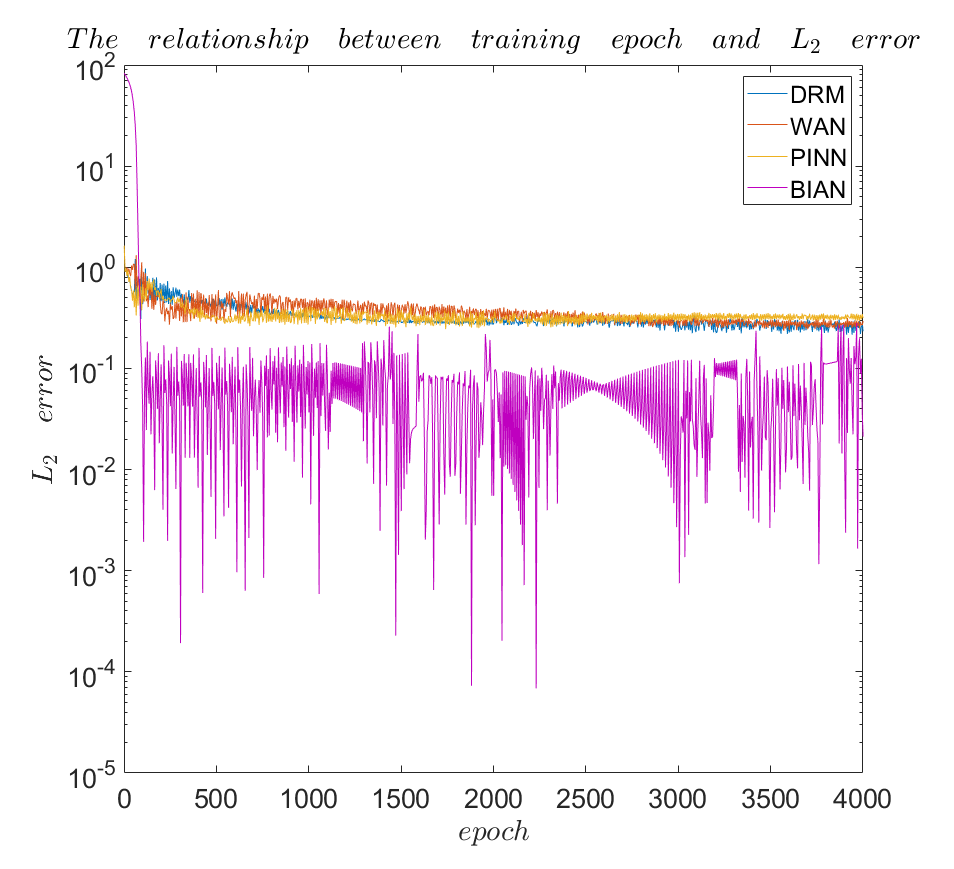}
      \caption{The relationship between the training epoch and the $L_2$ error of the solution.}
      \label{fig3_8}
\end{figure}

\subsection{The convergence analysis}
In this experiment, we aim to compare the advantages of BIAN over traditional methods in terms of convergence speed. Specifically, we focus on evaluating how BIAN performs in scenarios with varying medium distributions. We will analyze the convergence rates for BIAN and benchmark them against traditional approaches to highlight the efficiency gains that BIAN provides in solving complex inverse problems. Through this comparison, we aim to demonstrate the superior convergence performance of BIAN in handling high-dimensional and non-linear cases, where traditional methods may struggle to achieve the same level of accuracy and computational efficiency.

The problem to be solved is this experiment, as well as the architecture of the neural network, are identical to those used in \Sref{sec4_1}. We focus on comparing the relationship between solution accuracy and the size of the training data for BIAN and PINN when applied to the same problem. Specifically, we aim to assess how the two methods scale in terms of accuracy as the amount of training data increases. By examining this relationship, we can gain insights into the data efficiency of BIAN compared to PINN, particularly in high-dimensional and complex scenarios. This comparison is crucial for understanding the trade-offs between data size and solution precision, and it highlights the potential advantages of BIAN in reducing the need for extensive training data while maintaining or improving accuracy. \Fref{fig3_15} shows the relationship between the number of training data and the $L_2$ error. As a reference, The $O(m^{-1.5})$ rata of convergence is plotted as a dotted line. As we can see, with the same number of training data, BIAN achieves higher accuracy compared with IPINN. Additionally, BIAN exhibits faster convergence rate compared with IPINN.  
\begin{figure}[h]
      \centering
      \includegraphics[scale=0.49]{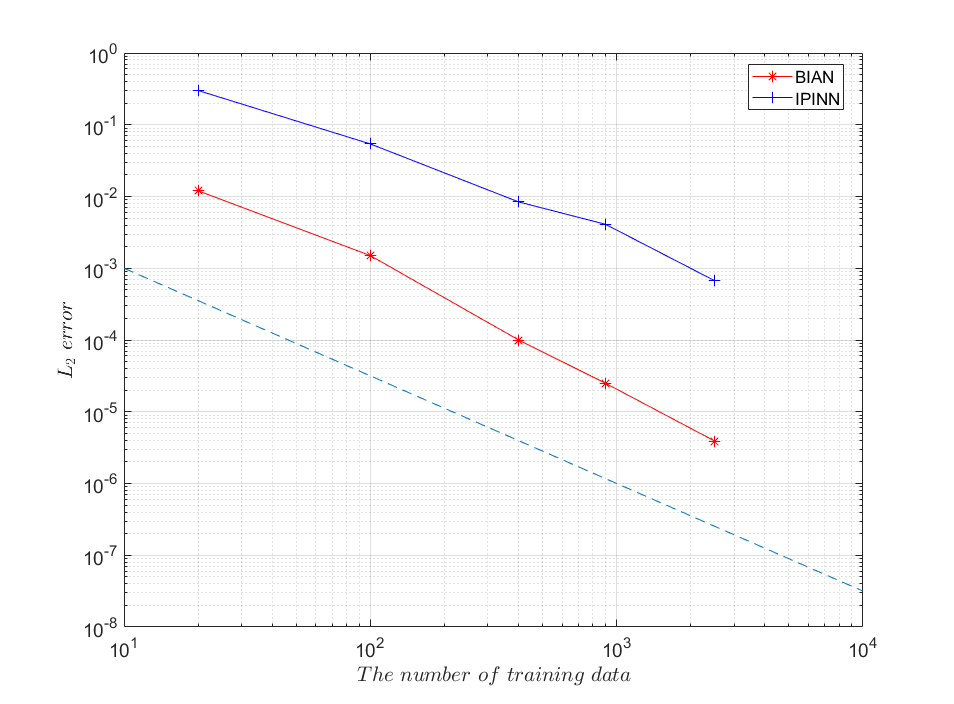}
      \caption{The relationship between the number of training data and the $L_2$ error of the solution.}
      \label{fig3_15}
\end{figure}

\subsection{Dealing with Piecewise uniform medium distribution}
The following experiment is designed to demonstrate the effectiveness of BIAN in solving problems with piecewise uniform medium distributions. By conducting this experiment, we aim to showcase the capability of BIAN to accurately approximate the unknown medium parameters and solve the associated partial differential equations under these specific conditions. The results will highlight the performance of BIAN in terms of computational accuracy and efficiency, further establishing its potential as a robust method for solving complex inverse problems.

Consider the Poisson equation with indeterminate piecewise uniform medium distribution:
\numparts
\begin{eqnarray}
     -\nabla \cdot (\varepsilon(\mathbf{x})\nabla u(\mathbf{x}) = f(\mathbf{x}), \quad &\mathbf{x}\in\Omega \label{eq3_2_1}\\
     u(\mathbf{x})=1,\quad &\mathbf{x}\in\Gamma_1 \label{eq3_2_2}\\
     u(\mathbf{x})=0,\quad &\mathbf{x}\in \Gamma_2 \label{eq3_2_3}
\end{eqnarray}
\endnumparts
Here $\Omega=\{(x,y)|(x,y)\in(0,1)\times(0,1)/(0,0.5)\times(0.5,1) \}$, $\Gamma_1 = \{ (x,y)|(x,y)=(0,1)\times {0} \}$, $\Gamma_2 = \partial \Omega \backslash \Gamma_1$, $\varepsilon(\mathbf{x})=10$ when $\mathbf{x}=\{(x,y)|(x,y)=(0,1)\times (0,0.5)\}$, $\varepsilon(\mathbf{x})=5$ when $\mathbf{x}=\{(x,y)|(x,y)=(0.5,1)\times (0.5,1)\}$ and $f(x)=1$. However, the two kinds of boundary conditions are all needed when utilizing BIAN to solve the inverse problems, we obtain the natural boundary condition for this problem with the Finite difference Method (FDM). The solution and the medium distribution of the Laplace problem is shown as \Fref{fig3_3}, and we can get the analytical solution of the natural boundary condition. Each layer of the residual blocks in the approximator network and generator network we used to solve this problem has $m=10$ neurons, there are a total of parameters in this model. We select 10 evenly spaced points on each boundary as the boundary source points $\{ P_i\}$ and 100 random points inside the domain as the internal source point $\{ p_j\}$. The selection of the integration points is consistent with the source points.
\begin{figure}[htbp]
	\centering
	\begin{minipage}{0.49\linewidth}
		\centering
		\includegraphics[width=1\linewidth]{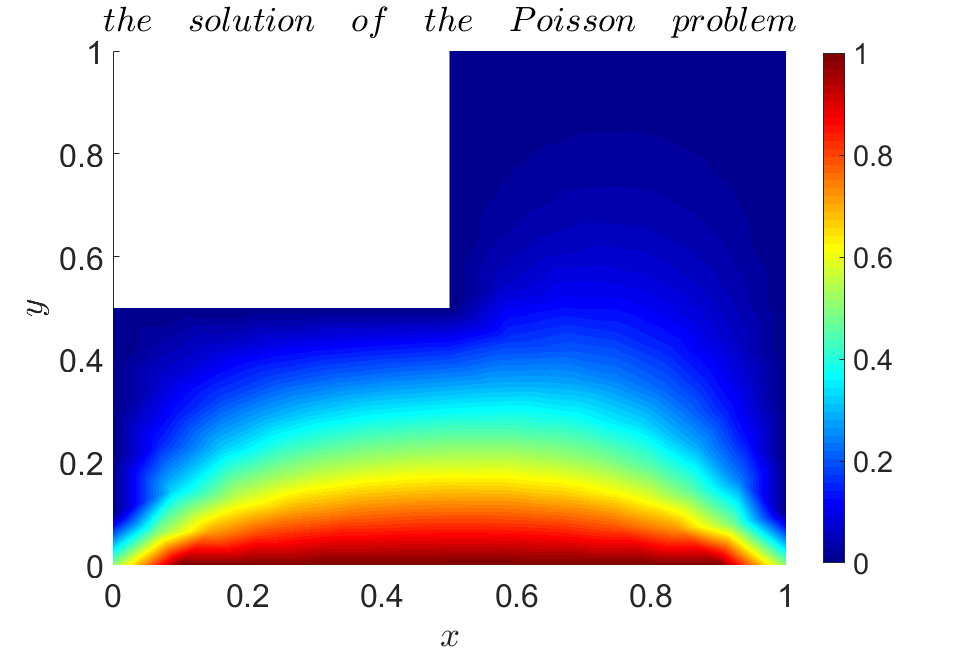}
		\label{fig3_9}
	\end{minipage}
	\begin{minipage}{0.49\linewidth}
		\centering
		\includegraphics[width=1\linewidth]{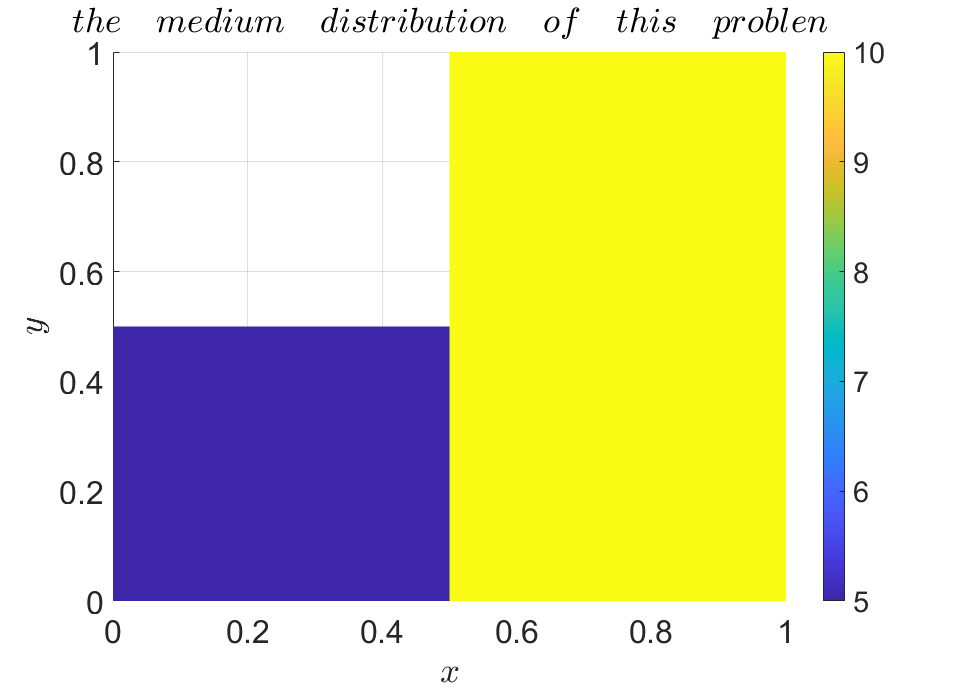}
		\label{fig3_10}
	\end{minipage}
      \caption{The solution and the medium distribution of the Poisson problem.}
\end{figure}
\begin{figure}[htbp]
	\centering
	\begin{minipage}{0.49\linewidth}
		\centering
		\includegraphics[width=1\linewidth]{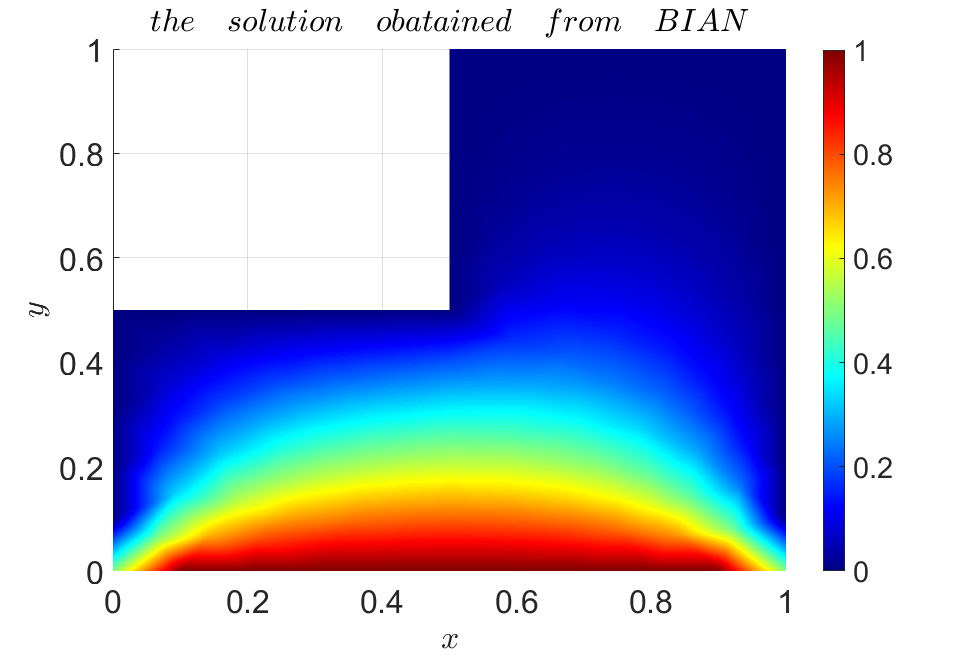}
		\label{fig3_11}
	\end{minipage}
	\begin{minipage}{0.49\linewidth}
		\centering
		\includegraphics[width=1\linewidth]{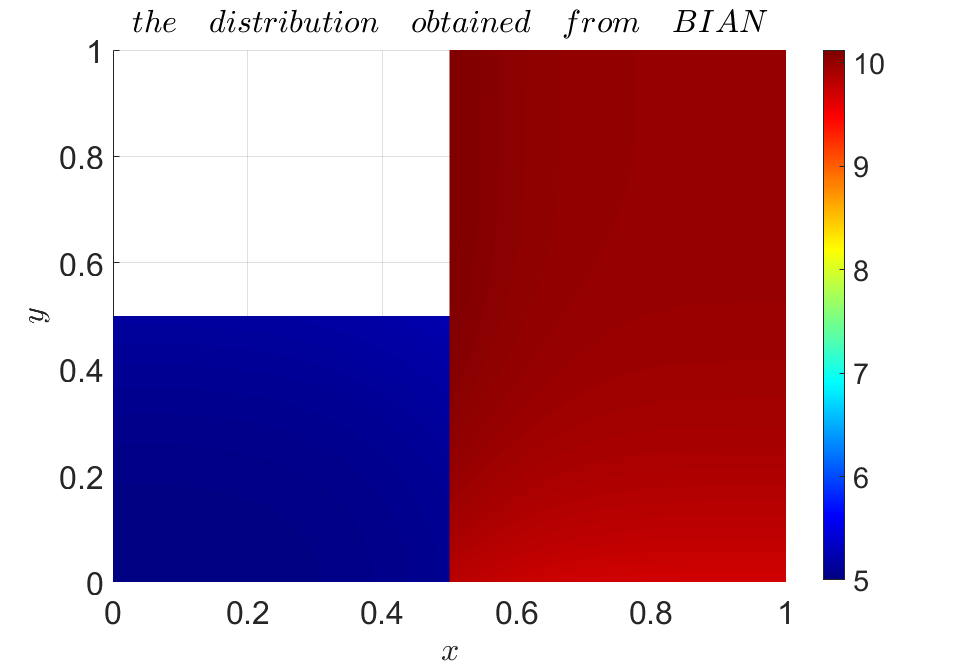}
		\label{fig3_12}
	\end{minipage}
      \caption{The solution and the medium distribution of the Poisson problem obtained from BIAN.}
\end{figure}
\begin{figure}[H]
	\centering
	\begin{minipage}{0.49\linewidth}
		\centering
		\includegraphics[width=1\linewidth]{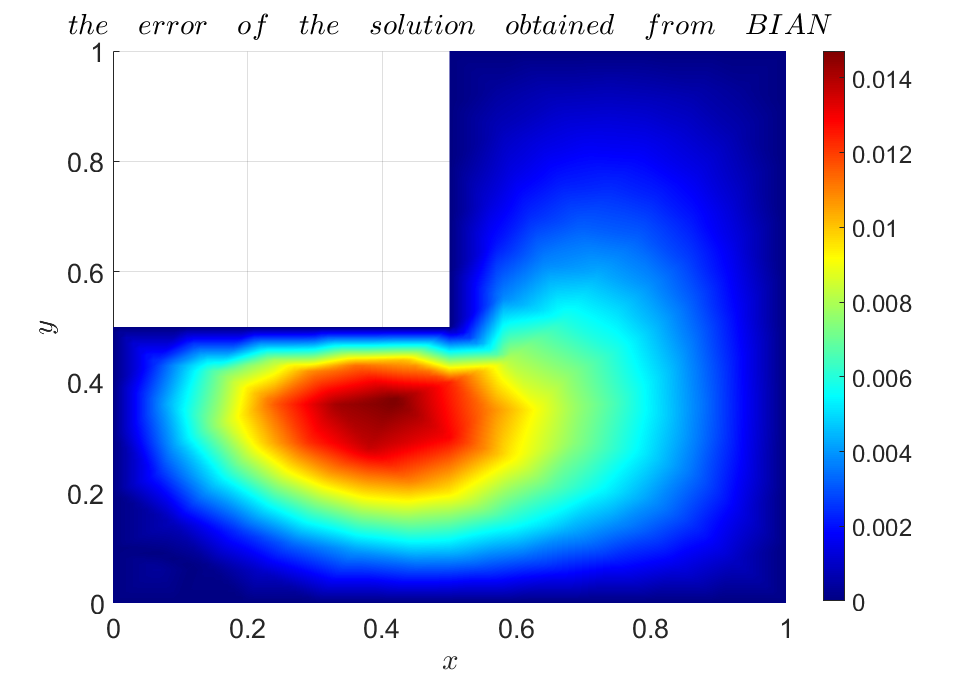}
		\label{fig3_13}
	\end{minipage}
	\begin{minipage}{0.49\linewidth}
		\centering
		\includegraphics[width=1\linewidth]{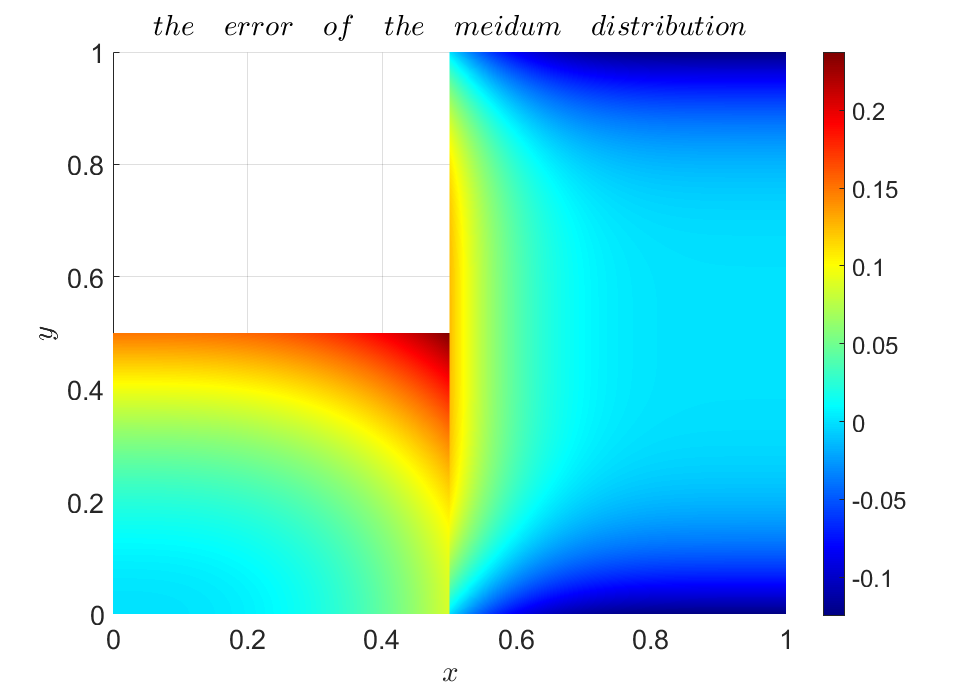}
		\label{fig3_14}
	\end{minipage}
      \caption{The error of the solution and the medium distribution of the Poisson problem.}
\end{figure}

Each layer of the residual blocks in the approximator network and generator network we used to solve this problem has $m=10$ neurons, there are a total of parameters in this model. We select 10 evenly spaced points on each boundary as the boundary source points $\{ P_i\}$ and 100 random points inside the domain as the internal source point $\{ p_j\}$. The selection of the integration points is consistent with the source points. 
The solution and the medium distribution approximated by BIAN is shown as \Fref{fig3_11}. The $L_2$ errors of the solution and medium distribution are $0.0079$ and $0.069$, respectively. From this numerical experiment, it proves that BIAN can not only deal with spatially varying medium distribution problems but also solve the piecewise uniform medium distribution problems.

\subsection{Dealing with high-dimensional problem}
To demonstrate the capability of the BIAN in addressing high-dimensional and geometrically complex problems, this section focuses on solving a three-dimensional anisotropic diffusion equation. We consider the partial differential equation defined on the unit cubic domain $\Omega=[0,1]^3$:
\begin{eqnarray}
     \nabla \cdot (\epsilon(\mathbf{x})\nabla u(\mathbf{x}))=10, \quad \mathbf{x}\in \Omega  \label{eq3_3_1}
\end{eqnarray}
with homogeneous Dirichlet boundary conditions:
\begin{eqnarray}
	u(\mathbf{x})=0,\quad \mathbf{x}\in \partial \Omega \label{eq3_3_2}
\end{eqnarray}
where $\epsilon(\mathbf{x})$ donates the spatially varying media distribution, in this problem $\epsilon(\mathbf{x})=\sin(x_1)+\cos(x_2)+x_3$. In contrast to traditional approaches that rely on internal node information, BIAN is immune to the curse of dimensionality. 

To address the challenges of solving high-dimensional problems, we . Each layer of the residual blocks in the approximator network and generator network we used to solve this problem has $m=10$ neurons, there are a total of parameters in this model. We select 10 evenly spaced points on each boundary as the boundary source points $\{ P_i\}$ and 100 random points inside the domain as the internal source point $\{ p_j\}$. The selection of the integration points is consistent with the source points. 
	
\begin{figure}[H]
      \centering
      \includegraphics[scale=0.6]{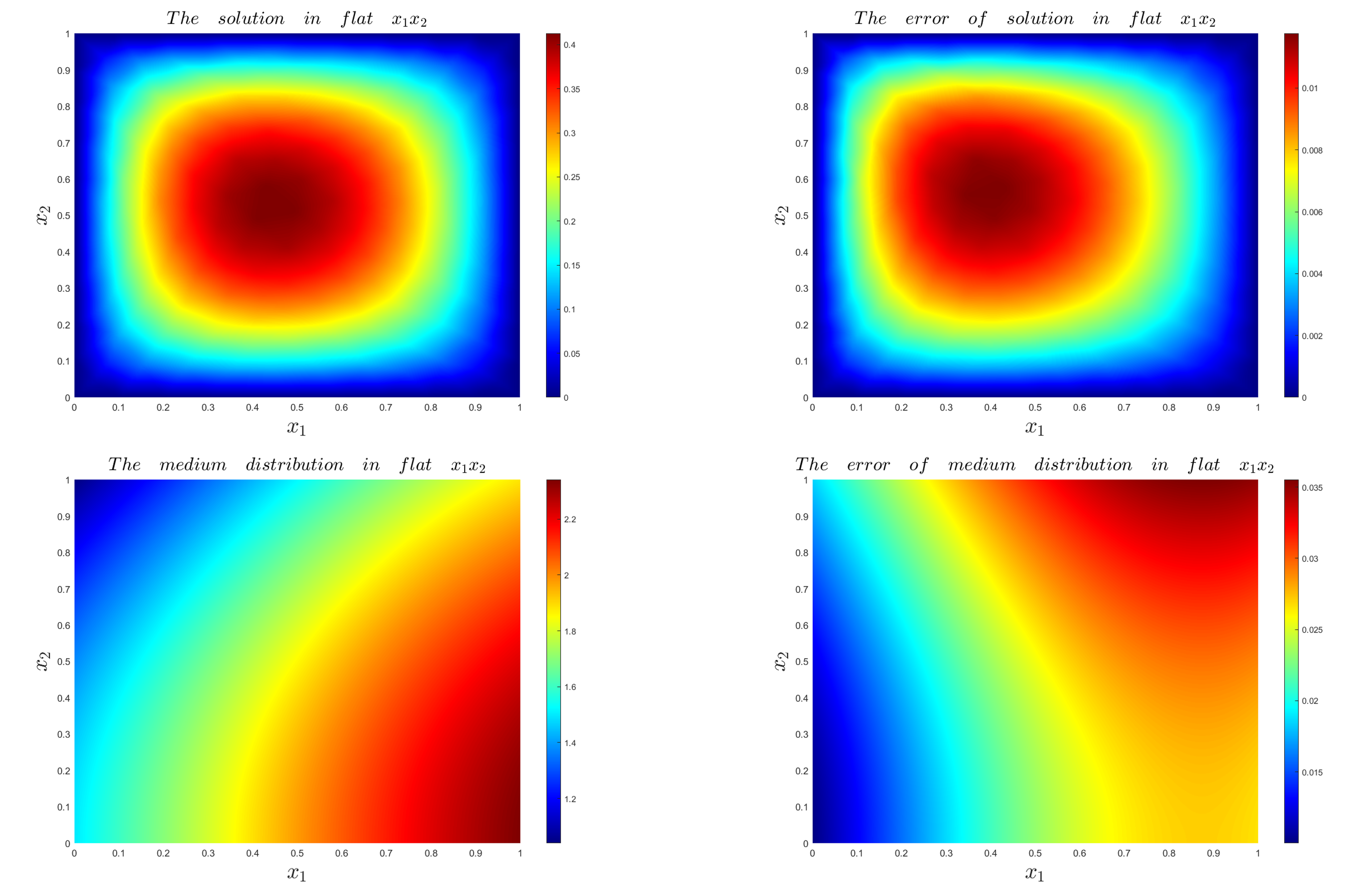}
      \caption{The results and error distributions of the solution function and medium distribution on the $x_1x_2$ plane obtained by BIAN.}
      \label{fig3_15}
\end{figure}

\begin{figure}[H]
      \centering
      \includegraphics[scale=0.6]{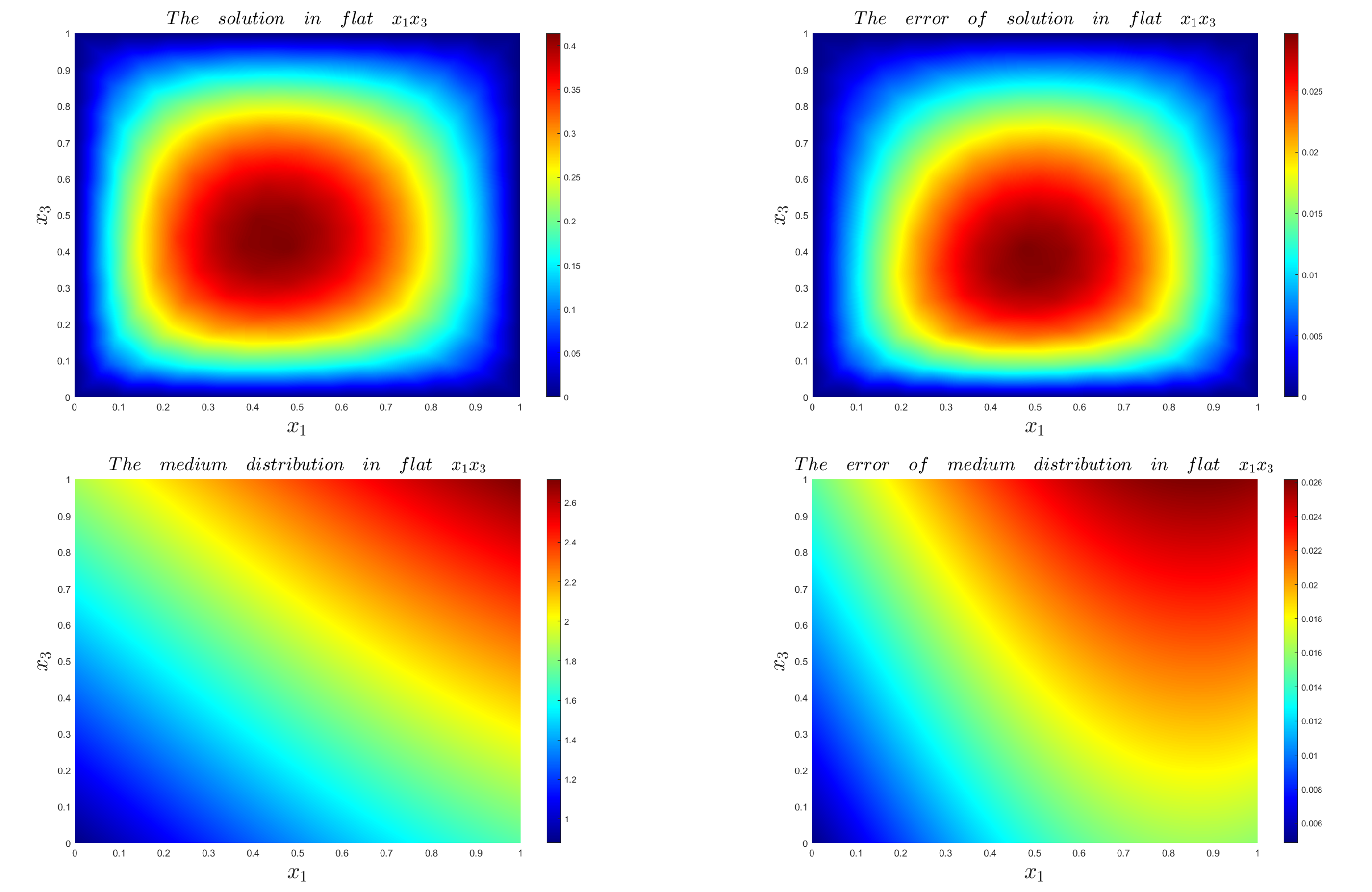}
      \caption{The results and error distributions of the solution function and medium distribution on the $x_1x_3$ plane obtained by BIAN.}
      \label{fig3_16}
\end{figure}

\begin{figure}[H]
      \centering
      \includegraphics[scale=0.6]{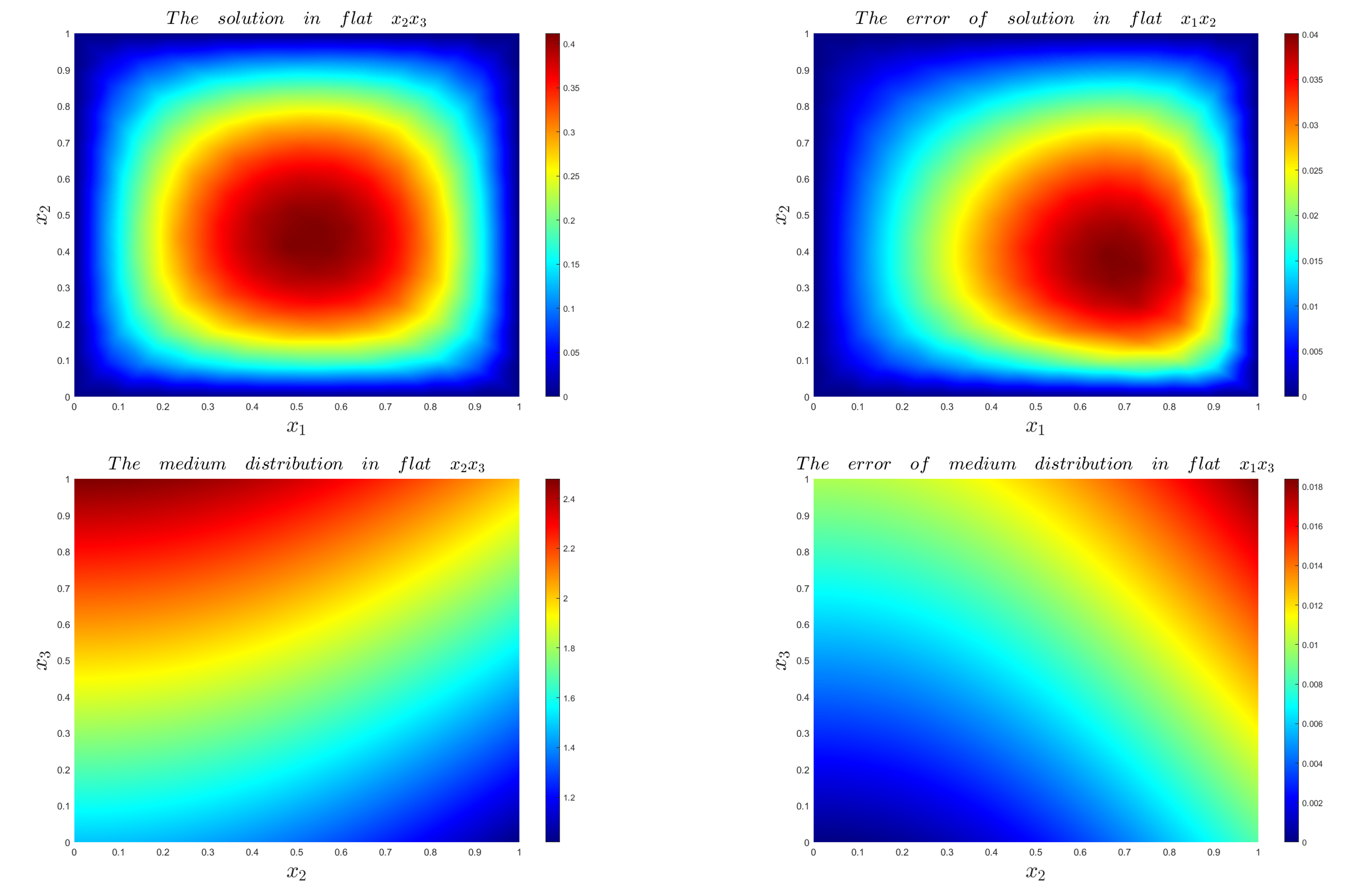}
      \caption{The results and error distributions of the solution function and medium distribution on the $x_2x_3$ plane obtained by BIAN.}
      \label{fig3_17}
\end{figure}

\begin{figure}[H]
      \centering
      \includegraphics[scale=0.4]{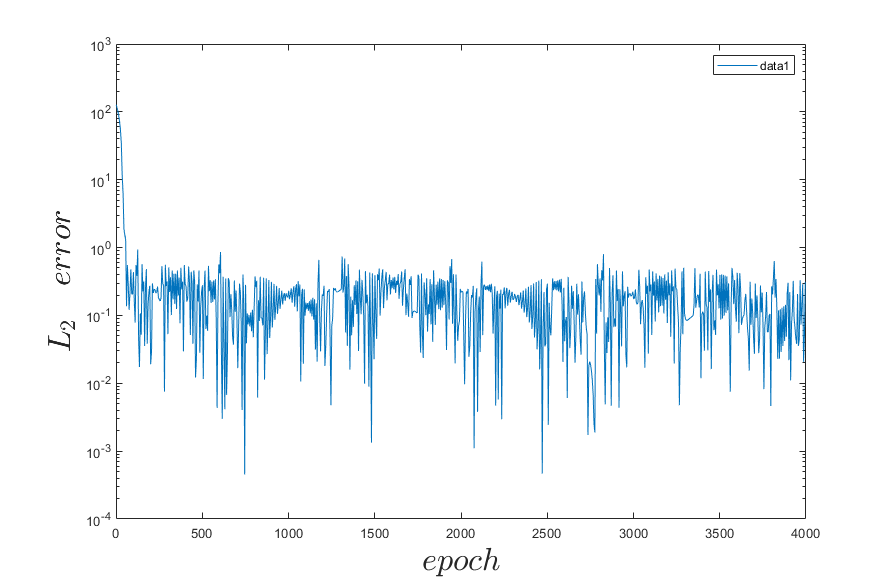}
      \caption{The relationship between the training epoch and the $L_2$ error of the solution.}
      \label{fig3_18}
\end{figure}
	
A comparative analysis of the solution functions and medium error distributions obtained by BIAN on the$x_1x_2$, $x_1x_3$ and $x_2x_3$ planes is presented in \Fref{fig3_15}, \Fref{fig3_16}, and \Fref{fig3_17}, respectively. These results are compared against the numerical solutions to evaluate the accuracy of BIAN. The evolution of the $L_2$ error with respect to the number of training iterations is illustrated in \Fref{fig3_18}.

As shown in \Fref{fig3_18}, the computational efficiency of BIAN remains robust even as the dimensionality increases. Specifically, BIAN achieves relatively accurate results within 500 training iterations for three-dimensional problems, demonstrating comparable performance to its two-dimensional counterpart. However, a slight degradation in solution accuracy is observed in higher dimensions, which can be attributed to the impact of dimensionality on the convergence rate of BIAN. Under the same data volume, the training performance of BIAN gradually declines as the dimensionality increases. Additionally, the oscillations in the $L_2$ error during training are likely caused by the random sampling of source points within the domain.
	
\section{Conclusion}
\label{section4}
In this work, we proposed a novel method based on representing medium coefficients with a neural network to solve inverse problems. This method, Boundary-Informed Alone Neural Network (BIAN), offers several significant advantages:
\begin{itemize}
    \item {} 
    With the incorporation of Green's theorem of energy conservation, the relationship between boundary energy flux and the energy distribution within the solution
    domain has be established. There is no need for inner points when training the neural network, which provides significant assistance in addressing practical issues. Moreover, measurement points can be sampled relatively uniformly along the boundary, mitigating the issue of non-uniform sample data distribution. 
    \item {} 
    The method exhibits strong performance in approximating complex functions and solving high-dimensional problems due to the high-dimensional capabilities of neural networks.
    \item {} 
    With the collaborative training between three neural networks, both training accuracy and convergence speed are enhanced compared to when the two neural networks are trained independently.
    \item{}
    This method is capable of achieving higher accuracy under the same training data. Moreover, it demonstrates a faster convergence rate, making it more efficient in solving complex problems. The combination of improved precision and quicker convergence not only reduces computational time but also enhances the overall performance, particularly in scenarios where data availability is limited or problem complexity increases.
\end{itemize}

However, we also some disadvantages that could be addressed in future works.
\begin{itemize}
    \item {} 
    The integral term in the loss function requires computation via Gaussian integration, which may lead to a reduction in solution accuracy.
    \item {} 
    When dealing with piecewise medium problems, the position of the interface between the two media is essential. Inferring the position of medium distribution from data may be impractical
\end{itemize}

Future research could explore the introduction of new neural network architectures, novel optimization algorithms, and advanced activation functions to potentially yield superior results. The numerical experiments conducted demonstrate BIAN's superior computational accuracy and efficiency compared to existing methods such as PINN, DRM, and WAN, particularly in solving problems with piecewise uniform medium distributions. This work highlights the potential of BIAN as a robust method for solving complex inverse problems, paving the way for its application in various fields such as engineering and physical sciences.

\appendix
\section{Proof of Theorem 1}
The proof consists two lemmas.
\begin{lemma}
      Let $\mathcal{T}_{b}^{m_b}=\{\mathbf{x}_{b}^{i}\}_{i=1}^{m_{b}}$ and $\mathcal{T}_{i}^{m_i}=\{\mathbf{x}_{i}^{j}\}_{j=1}^{m_{i}}$. Support $m_b$ and $m_i$ are large enough to satisfy the following: for any $x_b \in \Gamma$ and $x_i \in U$, there exists $x'_b \in \mathcal{T}_{b}$ and $x'_i \in \mathcal{T}_{i}$ such that $||x_b - x'_b||\leq \epsilon_r$ and $||x_i - x'_i||\leq \epsilon_i$. Then we can get:
      \begin{eqnarray}
      \label{eq6_1}
      \mathrm{Loss}(h) \leq C_1\mathrm{Loss}_m(h) + C_{max}(\epsilon_b^{2a}+\epsilon_b^{2a+d-1} +\epsilon_i^{2a+d})
      \end{eqnarray}
      where $C_1=3 C_{b}C_{i}m_b m_i \epsilon_{b}^{d-1} \epsilon_{i}^{d}$ and $C_{max} = max\{2[m]^2_{a,U}, 2[n]^2_{a,U}, 2[z]^2_{a,U}\}$.
\end{lemma}
\begin{proof}
      Let first consider the extended form of Cauchy-Schwarz inequality $\|x+y+z\|^2\leq3(\|x\|^2+\|y\|^2+\|z\|^2)$. For $x'_b, x'_i \in U$, we have
      \begin{eqnarray}
      \label{eq6_2}
      L(x_{b},x_{i};h)&=||F(x_{b})-A_{i}h(x_{i})||^{2} \nonumber\\&=||F(x_{b})-F(x'_{b})+F(x'_{b})-A[h](x'_{i})+A[h](x'_{i})-A[h](x_{r})||^{2} \nonumber\\&\leq3[||F(x_{b})-F(x'_{b})||^{2}+||F(x'_{b})-A[h](x'_{i})||^{2}\nonumber \\&+||A[h](x'_{i})-A[h](x_{i})||]^{2}
      \end{eqnarray}
      By assumption, $\forall x_b\in \Gamma$ and $\forall x_i \in U$, there exists $x'_b \in \mathcal{T}_{b}^{m_b}$ and $x'_i \in \mathcal{T}_{i}^{m_i}$ satisfied $||x_b-x'_b||\leq \epsilon_b$ and $||x_i-x'_i||\leq \epsilon_i$. Then we can get
      \begin{eqnarray}
      \label{eq6_3}
      ||F(x_{b})-F(x'_{b})||^2&=||m(x_b)-m(x'_b) + \int_\Gamma n(x_b)d\Gamma-\int_\Gamma n(x'_b)d\Gamma||^2 \nonumber \\ &\leq  2\int_\Gamma ||n(x_b)-n(x'_b)||^2d\Gamma \nonumber \\ &\leq  2\epsilon_b^{d-1}\epsilon_b^{2a}[n]^2_{a,U}
      \end{eqnarray}
      \begin{eqnarray}
      \label{eq6_4}
      ||A[h](x'_{i})-A[h](x_{i})||]^{2}&=||\int_\Omega z(x_i)d\Omega - \int_\Omega z(x'_i)d\Omega||^2 \nonumber \\&\leq \int_\Omega ||z(x_i)- z(x'_i)||^2 d\Omega \nonumber \\&\leq \epsilon_i^d\epsilon_i^{2a} [z]^2_{a,U}
      \end{eqnarray}
      \begin{eqnarray}
      \label{eq6_5}
      L(x_{b},x_{i};h)&\leq 3L(x'_{b},x'_{i};h) +  2\epsilon_b^{2a+d-1}[n]^2_{a,U} + \epsilon_i^{2a+d} [z]^2_{a,U}
      \end{eqnarray}
      For $x_b^i\in \mathcal{T}_{b}^{m_b}$, $A_{x_b^i}$ is the Voroni cell associated with $x_b^i$, i.e.,
      \begin{eqnarray}
      \label{eq6_6}
      A_{\mathbf{x}_b^i}=\left\{\mathbf{x}\in U|\|\mathbf{x}-\mathbf{x}_b^i\|=\min_{\mathbf{x}^{\prime}\in\mathcal{T}_b^{m_b}}\|\mathbf{x}-\mathbf{x}^{\prime}\|\right\},\nonumber
      \end{eqnarray}
      Similarly, let $x_i^j\in \mathcal{T}_{i}^{m_i}$, we have
      \begin{eqnarray}
      \label{eq6_7}
      A_{\mathbf{x}_i^j}=\left\{\mathbf{x}\in\Gamma|\|\mathbf{x}-\mathbf{x}_i^j\|=\min_{\mathbf{x}^{\prime}\in\mathcal{T}_i^{m_i}}\|\mathbf{x}-\mathbf{x}^{\prime}\|\right\}, \nonumber
      \end{eqnarray}
      Let $\gamma_{b}^{i}=\mu_{b}(A_{x_{b}^{i}})$ and $\gamma_{i}^{j}=\mu_{i}(A_{x_{i}^{j}})$ which satisfy $\sum_{i=1}^{m_b}\gamma_{b}^{i}=1$ and $\sum_{j=1}^{m_i}\gamma_{i}^{j}=1$. Taking the expectation with respect to $(x_b,x_i){\sim}=\mu=\mu_b\times\mu_i$, we have
      \begin{eqnarray}
      \label{eq6_8}
      \mathrm{Loss}(h) &\leq 3\sum_{i=1}^{m_{b}}\sum_{j=1}^{m_{i}}\gamma_{b}^{i}\gamma_{i}^{j}L(x_b^i,x_i^j;h) \nonumber\\ &+ 2\epsilon_b^{2a+d-1}[n]^2_{a,U} + \epsilon_i^{2a+d} [z]^2_{a,U}
      \end{eqnarray}
      With $\gamma_b^{m_b,*}=\max_i\gamma_b^i$ and $\gamma_i^{m_i,*}=\max_j\gamma_i^j$ we can obtain
      \begin{eqnarray}
      \label{eq6_9}
      \mathrm{Loss}(h) &\leq 3 m_b\gamma_b^{m_b,*} m_i \gamma_i^{m_i,*} \frac{1}{m_b}\frac{1}{m_i}\sum_{i=1}^{m_{b}}\sum_{j=1}^{m_{i}}\gamma_{b}^{i}\gamma_{i}^{j}L(x_b^i,x_i^j;h) \nonumber \\ &+  2\epsilon_b^{2a+d-1}[n]^2_{a,U} + \epsilon_i^{2a+d} [z]^2_{a,U}
      \end{eqnarray}
      Note that $m_b\gamma_b^{m_b,*}, m_i \gamma_i^{m_i,*}\geq 1.$ Let $B_\epsilon(\mathbf{x})$ be a closed ball centered at x with radius $\epsilon.$ Let $P_{b}^{* }= \max _{\mathbf{x} \in \Gamma}\mu _{b}\left ( B_{\epsilon _{b}}( \mathbf{x} ) \cap \Gamma\right )$and $P_{i}^{*}=\max_{\mathbf{x}\in  U}\mu_{i}(B_{\epsilon_{i}}(\mathbf{x})\cap U).$ Since for any $\mathbf{x}_b\in \Gamma$ and $\mathbf{x}_{i}\in U$, there exists $\mathbf{x}_b^\prime\in\mathcal{T}_{b}^{m_b}$ and $\mathbf{x}_i^\prime\in\mathcal{T}_{i}^{m_i}$ such that $\|\mathbf{x}_b-\mathbf{x}_b^{\prime}\|\leq\epsilon_b$, and $\|\mathbf{x}_i-\mathbf{x}_i^{\prime}\|\leq\epsilon_i$ for each $i$,there are closed balls $B_{\epsilon_b}$ and $B_{\epsilon_i}$ that include $A_{\mathbf{x}_b^i}$ and $A_{\mathbf{x}_i^i}$, respectively. Thus, we have $\gamma_b^{m_b,*}\leq P_b^*,\gamma_i^{m_i,*}\leq P_i^*.$ Moreover, it follows from Assumption \ref{as_1} that
      \begin{eqnarray}
      \label{eq6_10}
      \gamma_b^{m_b,*}\leq P_{b}^{*}\leq C_{b}\epsilon_{b}^{d-1},\quad\gamma_i^{m_i,*}\leq P_{i}^{*}\leq C_{i}\epsilon_{i}^{d}.
      \end{eqnarray}
      Therefore we can get
      \begin{eqnarray}
      \label{eq6_11}
      \mathrm{Loss}(h) \leq C_1\mathrm{Loss}_m(h) + C_{max}(\epsilon_b^{2a+d-1} +\epsilon_i^{2a+d})
      \end{eqnarray}
      where $C_1=3 C_{b}C_{i}m_b m_i \epsilon_{b}^{d-1} \epsilon_{i}^{d}$ and $C_{max} = max\{ 2[n]^2_{a,U}, 2[z]^2_{a,U}\}$. The proof is completed.
\end{proof}

\begin{lemma}
      \label{lemma_2}
     Let $X$ be a compact subset in $\mathbb{R}^{d}$. Let $\mu$ be a probability measure supported on $X$. Let $\rho$ be the probability density of $\mu$ with respect to s-dimensional Hausdorff measure on $X$ such that $\inf_ {X_\rho}>0$. Suppose that for $\epsilon >0$, there exists a partition of $X$, $\{X_{k}^{\epsilon}\}_{k=1}^{K_{\epsilon}}$ that depends on $\epsilon$ such that for each $X_{k}^{\epsilon}$, $c\epsilon^{s}\leq\mu(X_{k}^{\epsilon})$ where $c>0$ depends only on $(\mu,X)$, and there exists a cube $H_\epsilon(z_k)$ of side length $\epsilon$ centered at some $z_k$ in $X_k$ such that $X_k \subset H_\epsilon(z_k)$. Then, with probability
     at least $1-\sqrt{n}(1-1/\sqrt{n})^{n}$ over iid n sample points $\{x_i\}_{i=1}^n$ from $\mu$, for any $x \in X$, there exists a point $x_j$ such that $\|x-x_{j}\|\leq\sqrt{d}c^{-\frac{1}{s}}n^{-\frac{1}{2s}}$.
\end{lemma}
\begin{proof}
      The proof of this lemma has already been proven in Apppendix B of \cite{Shin2020OnTC}.
\end{proof}
By lemma \ref{lemma_2}, with the probability at least
\begin{eqnarray}
      \label{eq6_12}
      (1-\sqrt{m_b}(1-1/\sqrt{m_b})^{m_b})(1-\sqrt{m_i}(1-1/\sqrt{m_i})^{m_i}),
\end{eqnarray}
$\forall x_b \in \Gamma$ and $\forall x_i \in U$, there exists $\mathbf{x}_b^\prime\in\mathcal{T}_{b}^{m_b}$ and $\mathbf{x}_i^\prime\in\mathcal{T}_{i}^{m_i}$ satisfied $\|\mathbf{x}_{b}-\mathbf{x}_{b}^{\prime}\|\leq\sqrt{d}c_{b}^{-\frac{1}{d-1}}m_{b}^{-\frac{1}{2(d-1)}}$ and $||\mathbf{x}_{i}-\mathbf{x}_{i}^{\prime}\|\leq{\sqrt{d}}c_{i}^{-{\frac{1}{d}}}m_{i}^{-{\frac{1}{2d}}}$. By letting  $\epsilon_b = \sqrt{d}c_{b}^{-\frac{1}{d-1}}m_{b}^{-\frac{1}{2(d-1)}}$ and $ \epsilon_i = {\sqrt{d}}c_{i}^{-{\frac{1}{d}}}m_{i}^{-{\frac{1}{2d}}}$, with the probability at least \Eref{eq6_12}, we have
\begin{eqnarray}
\label{eq6_13}
\mathrm{Loss}(h) \leq C'_1\mathrm{Loss}_m(h) + C'_{max}(m_b^{-\frac{a}{d-1}-0.5} +m_i^{-\frac{a}{d}-0.5})
\end{eqnarray}
where $C'_1=3 \frac{C_{b}C_{i}}{c_b c_i}m_b^{0.5} m_i^{0.5} \sqrt d ^{2d-1}$ and $C_{max} = max\{ 2[n]^2_{a,U}d^{a+0.5d-0.5} c_b^{-\frac{2a+d-1}{d-1}}, \\2[z]^2_{a,U} d^{a+0.5d} c_b^{-\frac{2a+d}{d}}\}$. The proof is completed.

\section*{References}
\bibliographystyle{unsrt}
\bibliography{ref.bib}

\end{document}